\documentclass[journal,onecolumn]{IEEEtran}
\ifCLASSINFOpdf
\else
\fi

\usepackage{hyperref}  %hyperref still needs to be put at the end!
\usepackage{amsmath,amsfonts,amssymb}
 
\usepackage{amsthm}

\newtheorem{lemma}{Lemma}
\newtheorem{theorem}{Theorem}
\newtheorem{definition}{Definition}

\newtheorem{assumption}{Assumption}

\newtheorem{remark}{Remark}
\usepackage{cite}
\usepackage{algorithmic}
\usepackage{textcomp}
\usepackage{graphicx}
\usepackage{courier}
\usepackage{mathtools}
\usepackage{float}
\usepackage{multirow}
\DeclarePairedDelimiter{\norm}{\lVert}{\rVert}
\usepackage{mathrsfs}
\usepackage{accents}
\usepackage{float}

\DeclareMathOperator*{\argmin}{arg\,min}
\usepackage{cancel}

\begin{document}
	\title{Practical exponential stability of a robust data-driven nonlinear predictive	control scheme}
	\author{Mohammad Alsalti$^{\dag}$, Victor G. Lopez$^{\dag}$, Julian Berberich$^{\ddag}$, Frank Allgöwer$^{\ddag}$, and Matthias A. Müller$^{\dag}$ %
		\thanks{$^{\dagger}$Leibniz University Hannover, Institute of Automatic Control, 30167 Hannover, Germany. E-mail:\{\href{maitlo:alsalti@irt.uni-hannover.de}{alsalti},\href{maitlo:lopez@irt.uni-hannover.de}{lopez},\href{maitlo:mueller@irt.uni-hannover.de}{mueller}\}@irt.uni-hannover.de}%
		\thanks{$^{\ddagger}$University of Stuttgart, Institute for Systems Theory and Automatic Control, 70550 Stuttgart, Germany. E-mail: \{\href{maitlo:julian.berberich@ist.uni-stuttgart.de}{julian.berberich}, \href{maitlo:frank.allgower@ist.uni-stuttgart.de}{frank.allgower}\}@ist.uni-stuttgart.de}%
	}
	\maketitle
	\begin{abstract}
		We provide theoretical guarantees for recursive feasibility and practical exponential stability of the closed-loop system of a discrete-time multi-input multi-output full-state feedback linearizable nonlinear system when controlled by a robust data-driven nonlinear predictive control scheme. Two sources of uncertainties are considered: the mismatch due to the basis function approximation as well as output measurement noise. This technical report serves as a supplementary material to \cite{Alsalti2022b}.
	\end{abstract}
	\IEEEpeerreviewmaketitle
	\section{Introduction}\label{intro_sec}
In \cite{Alsalti2022a}, we considered a discrete-time multi-input multi-output (DT-MIMO) full-state feedback linearizable nonlinear systems and provided a data-based representation of their trajectories using a single, persistently exciting input-output trajectory and a set of basis functions that only depend on input and noisy output data. In general, the combination of basis functions serves as an approximation and not as an exact decomposition of the unknown nonlinearities. Furthermore, the collected output data is generally contaminated by noise. Therefore, we proposed in \cite[Section 3.2]{Alsalti2022b} a robust data-driven nonlinear predictive control scheme that stabilizes a fixed set-point of a DT-MIMO full-state feedback linearizable system despite unknown, but uniformly bounded, basis function approximation error and output noise. This supplementary report provides the theoretical analysis of recursive feasibility and practical exponential stability of this scheme. The proof is conceptually similar to the one presented in \cite{Berberich203} for DT-LTI systems. This is because feedback linearizable systems can be expressed as linear systems in transformed coordinates. However, additional challenges imposed by the nonlinear nature of the problem as well as having two sources of uncertainties require careful investigation which we provide in this report.
%%%%%%%%%%%%%%%%%%%%%%%%%%%%%%%%%%%%%%%%%%%%
\section{Notation}
The set of integers in the interval $[a,b]$ is denoted by $\mathbb{Z}_{[a,b]}$. For a vector $\mu\in\mathbb{R}^n$ and a positive definite symmetric matrix $P=P^\top\succ0$, the $p-$norm is given by $\norm*{\mu}_p$ for $p=1,2,\infty$, whereas $\norm*{\mu}_P = \sqrt{\mu^\top P\mu}$. The minimum and maximum singular values of the matrix $P$ are denoted by $\sigma_{\textup{min}}(P),\sigma_{\textup{max}}(P)$, respectively. Similarly, $\lambda_{\textup{min}}(P),\lambda_{\textup{max}}(P)$ denote the minimum and maximum eigenvalues of that matrix. The induced norm of a matrix $P$ is denoted by $\norm*{P}_i$ for $i=1,\infty$. Wherever necessary, we use $\mathbf{0}$ to denote a vector or matrix of zeros of appropriate dimensions. For a sequence $\{\mathbf{z}_k\}_{k=0}^{N-1}$ with $\mathbf{z}_k\in\mathbb{R}^\eta$, each element is expressed as $\mathbf{z}_k=\begin{bmatrix}
	z_{1,k} & z_{2,k} & \dots & z_{\eta,k}
\end{bmatrix}^\top$. The stacked vector of that sequence is given by $\mathbf{z}=\begin{bmatrix}
	\mathbf{z}_0^\top & \dots & \mathbf{z}_{N-1}^\top
\end{bmatrix}^\top$, and a window of it by $\mathbf{z}_{[a,b]}=\begin{bmatrix}
	\mathbf{z}_a^\top & \dots & \mathbf{z}_b^\top
\end{bmatrix}^\top$. The Hankel matrix of depth $L$ of this sequence is given by
\begin{equation*}
	\begin{aligned}
		H_L(\mathbf{z})&=\begin{bmatrix}
			\mathbf{z}_0 & \mathbf{z}_1 & \dots & \mathbf{z}_{N-L}\\
			\mathbf{z}_1 & \mathbf{z}_2 & \dots & \mathbf{z}_{N-L+1}\\
			\vdots & \vdots & \ddots & \vdots\\
			\mathbf{z}_{L-1} & \mathbf{z}_L & \dots & \mathbf{z}_{N-1}
		\end{bmatrix}.
	\end{aligned}
\end{equation*}
Throughout the report, the notion of persistency of excitation (PE) is defined as follows.
\begin{definition}
	The sequence $\{\mathbf{z}_k\}_{k=0}^{N-1}$ is said to be persistently exciting of order $L$ if \textup{rank}$\left(H_L(\mathbf{z})\right)=\eta L$.
\end{definition}
%%%%%%%%%%%%%%%%%%%%%%%%%%%%%%%%%%%%%%%%%%%%
\section{Preliminaries}
Consider a DT-MIMO nonlinear system of the form
\begin{equation}
	\begin{aligned}
		\mathbf{x}_{k+1} &= \boldsymbol{f}(\mathbf{x}_k,\mathbf{u}_k),\\
		\mathbf{y}_k &= \boldsymbol{h}(\mathbf{x}_k),
	\end{aligned}
	\label{NLsys}
\end{equation}
with $\mathbf{x}_k\in\mathbb{R}^n,\,\mathbf{u}_k,\mathbf{y}_k\in\mathbb{R}^m$, being the state, input and output vectors at time $k$ respectively, and  $\boldsymbol{f}:\mathbb{R}^n\times\mathbb{R}^m\to\mathbb{R}^n,\,\boldsymbol{h}:\mathbb{R}^n\to\mathbb{R}^m$, with $\boldsymbol{f}(\mathbf{0},\mathbf{0})=\mathbf{0},\,\boldsymbol{h}(\mathbf{0})=\mathbf{0}$, being analytic functions. As defined in \cite{Monaco87}, each output $y_i=h_i(\mathbf{x})$ of the nonlinear system \eqref{NLsys}, for $i\in\mathbb{Z}_{[1,m]}$, is said to have a (globally) well-defined relative degree $d_i$ if at least one of the $m$ inputs at time $k$ affects the $i-$th output at time $k+d_i$. In particular,
\begin{equation}
	\begin{aligned}
		y_{i,k+d_i} &= h_i(\boldsymbol{f}_O^{d_i-1}\left(\boldsymbol{f}(\mathbf{x}_k,\mathbf{u}_k)\right)).
		\label{output_k}
	\end{aligned}
\end{equation}
Under suitable assumptions (cf. \cite[Assumptions 1-3]{Alsalti2022a}), there exists a feedback linearizing control law $\mathbf{u}_k=\gamma(\mathbf{x}_k,\mathbf{v}_k)$ and an invertible coordinate transformation $\Xi_k=T(\mathbf{x}_k)$ such that the system \eqref{NLsys} is linear in transformed coordinates \cite{Monaco87}. Specifically \eqref{NLsys} is transformed into the form,
\begin{equation}
	\begin{aligned}
		\Xi_{k+1} &= \mathcal{A}\Xi_k + \mathcal{B}\mathbf{v}_k,\\
		\mathbf{y}_k &= \mathcal{C}\Xi_k,
	\end{aligned}\label{Lsys}
\end{equation}
where $\mathcal{A,B,C}$ are in the block-Brunovsky canonical form (which are a controllable and observable triplet) and
\begin{equation}
	\Xi_k = \begin{bmatrix}y_{1,[k,k+d_1-1]}^\top & \dots & &y_{m,[k,k+d_m-1]}^\top\end{bmatrix}^\top,\label{definition_of_plain_Xi}
\end{equation}
with $d_i$ denoting the relative degree of the $i-$th output for $i\in\mathbb{Z}_{[1,m]}$ (see \cite{Monaco87,AlsaltiBerLopAll2021}). The block-Brunovsky form is defined as
\begin{equation*}
	\begin{matrix}
		\mathcal{A}\coloneqq\begin{bmatrix}
			A_1 & \dots & 0\\
			\vdots & \ddots & \vdots\\
			0 & \dots & A_m
		\end{bmatrix},
		&\mathcal{B}\coloneqq \begin{bmatrix}
			B_1& \dots & 0\\
			\vdots & \ddots & \vdots\\
			0 & \dots &B_m
		\end{bmatrix},
		&\mathcal{C}\coloneqq\begin{bmatrix}
			C_1& \dots & 0\\
			\vdots & \ddots & \vdots\\
			0 & \dots &C_m
		\end{bmatrix},
	\end{matrix}
\end{equation*}
with each $A_i,B_i,C_i$ for $i\in\mathbb{Z}_{[1,m]}$ defined as
\begin{equation*}
	\begin{matrix}
		A_i \coloneqq \begin{bmatrix}
			0&1&\dots&0 \\ \vdots&\ddots&\ddots&\vdots \\ \vdots& &\ddots&1 \\ 0&\dots&\dots&0
		\end{bmatrix}, &
		B_i \coloneqq \begin{bmatrix}
			0\\ \vdots\\ 0\\ 1
		\end{bmatrix}, &
		\,C_i \coloneqq \begin{bmatrix}
			1 & 0 & \dots & 0
		\end{bmatrix}.
	\end{matrix}
\end{equation*}
In \cite{Monaco87}, it was shown that the synthetic input $\mathbf{v}_k$ is given by the iterated composition of the analytic functions $\boldsymbol{f}(\cdot,\cdot),\,\boldsymbol{h}(\cdot)$ (and hence locally Lipschitz continuous) of the form
\begin{equation}
	\mathbf{v}_k = \begin{bmatrix}
		v_{1,k} \\ \vdots \\ v_{m,k}
	\end{bmatrix} = \begin{bmatrix}
		{h}_1\left(\boldsymbol{f}_O^{d_1-1}\left(\boldsymbol{f}(\mathbf{x}_k,\mathbf{u}_k)\right)\right)\\
		\vdots\\
		{h}_m\left(\boldsymbol{f}_O^{d_m-1}\left(\boldsymbol{f}(\mathbf{x}_k,\mathbf{u}_k)\right)\right)
	\end{bmatrix}\label{definition_of_v_and_Phi}.
\end{equation}
Since $\boldsymbol{f}(\cdot,\cdot),\,\boldsymbol{h}(\cdot)$ are unknown then $\mathbf{v}_k$ is also unknown. We parameterize this unknown nonlinear function by $\tilde{\Phi}(\mathbf{u}_k,\mathbf{x}_k):\mathbb{R}^m\times\mathbb{R}^n\to\mathbb{R}^m$. Notice that since $\mathbf{x}_k = T^{-1}(\Xi_k)$, one can write\footnote{Notice that $\Phi$ is the iterated composition of continuously differentiable functions $\boldsymbol{f},\boldsymbol{h}$ and $T^{-1}$ (which is continuously differentiable by \eqref{output_k} and \eqref{definition_of_plain_Xi}). Hence, $\Phi$ is locally Lipschitz continuous.} $\Phi(\mathbf{u}_k,\Xi_k)\coloneqq\tilde{\Phi}(\mathbf{u},T^{-1}(\Xi))$ and use a set of basis functions to approximate it. In particular,
\begin{equation}
	\mathbf{v}_k \eqqcolon \Phi(\mathbf{u}_k,\Xi_k) = \begin{bmatrix}
		\phi_1(\mathbf{u}_k,\mathbf{x}_k)\\
		\vdots\\
		\phi_m(\mathbf{u}_k,\mathbf{x}_k)
	\end{bmatrix} = \overbrace{\begin{bmatrix}
			\rule[.5ex]{1.5em}{0.4pt} \,g_1^\top\, \rule[.5ex]{1.5em}{0.4pt}\\
			\vdots\\
			\rule[.5ex]{1.5em}{0.4pt} \,g_m^\top\, \rule[.5ex]{1.5em}{0.4pt}\\
	\end{bmatrix}}^{\mathcal{G}}\Psi(\mathbf{u}_k,\Xi_k) + \overbrace{\begin{bmatrix}
			\varepsilon_1(\mathbf{u}_k,\Xi_k)\\ \vdots\\ \varepsilon_m(\mathbf{u}_k,\Xi_k)
	\end{bmatrix}}^{\scalebox{1.5}{$\epsilon$}(\mathbf{u}_k,\Xi_k)},
	\label{basis}
\end{equation}
where $\Psi(\mathbf{u}_k,\Xi_k)$ is the stacked vector of $r\in\mathbb{N}$ locally Lipschitz continuous (in $\Xi$) and linearly independent basis functions, $\scalebox{1.5}{$\epsilon$}(\mathbf{u}_k,\Xi_k)$ is the stacked vector of basis functions approximation errors and $\mathcal{G}$ is the matrix of \textit{unknown}\footnote{Notice that we do not evaluate the minimization problem in \eqref{G_def} throughout the report.} coefficients which can be defined in the following manner
\begin{equation}
	\mathcal{G}\coloneqq\argmin\limits_{{G}} \left< \Phi \hspace{-0.5mm}-\hspace{-0.5mm} {G}\Psi , \Phi \hspace{-0.5mm}-\hspace{-0.5mm} {G}\Psi \right>,
	\label{G_def}
\end{equation}
with $\left< \rho_1, \rho_2 \right> = \int_\Omega \rho_1(s_1,s_2)\rho_2(s_1,s_2) ds_1ds_2$ for some $\Omega\subseteq \mathbb{R}^m\times\mathbb{R}^n$.
Similar to \cite{Alsalti2022b}, the matrix $\mathcal{G}$ is assumed to have full row rank and, hence, a right inverse $\mathcal{G}^\dagger = \mathcal{G}^\top(\mathcal{GG}^\top)^{-1}$ exists.
In practice, the collected output measurements are noisy. In what follows, we consider random but uniformly bounded additive output measurement noise. This is denoted by $\tilde{\mathbf{y}}_k = \mathbf{y}_k + \mathbf{w}_k$, where $\norm*{\mathbf{w}_k}_\infty\leq w^*$ for all $k\geq0$. The corresponding state vector (see \eqref{definition_of_plain_Xi}) constructed from noisy data is denoted $\tilde{\Xi}_k = \Xi_k + \omega_k$ where $\omega_k=\begin{bmatrix}
	w_{1,[k,k+d_1-1]}^\top&\cdots&w_{m,[k,k+d_m-1]}^\top
\end{bmatrix}^\top$. Using the noisy data to approximate the unknown nonlinear function in \eqref{basis} results in an additional error term. In particular,
\begin{equation}
	\Phi(\mathbf{u}_k,\Xi_k) = \mathcal{G}\Psi(\mathbf{u}_k,\tilde{\Xi}_k) + \scalebox{1.5}{$\epsilon$}(\mathbf{u}_k,\tilde{\Xi}_k) + \delta(\omega_k),
	\label{noisy_basis}
\end{equation}
where $\delta(\omega_k)\coloneqq \mathcal{G}\Psi(\mathbf{u}_k,{\Xi}_k) + \scalebox{1.5}{$\epsilon$}(\mathbf{u}_k,{\Xi}_k) - \mathcal{G}\Psi(\mathbf{u}_k,\tilde{\Xi}_k) - \scalebox{1.5}{$\epsilon$}(\mathbf{u}_k,\tilde{\Xi}_k)$. Substituting \eqref{noisy_basis} back in \eqref{Lsys} results in
\begin{align}
	\Xi_{k+1} &= \mathcal{A}\Xi_k + \mathcal{B}\mathcal{G}(\hat{\Psi}_k(\mathbf{u},{\tilde{\Xi}})+\hat{E}_k(\mathbf{u},{\tilde{\Xi}}) + \hat{D}_k(\omega)),\notag\\
	{\mathbf{\tilde{y}}_k }&{= \mathcal{C}\Xi_k + \mathbf{w}_k},	\label{Lsys3}
\end{align}
where
\begin{equation}
	\begin{matrix}
		\hat{\Psi}_k(\mathbf{u},{\tilde{\Xi}}) \coloneqq \Psi(\mathbf{u}_k,{\tilde{\Xi}_k}), \quad
		\hat{\varepsilon}_{i,k}(\mathbf{u},{\tilde{\Xi}}) \coloneqq 	\varepsilon_i(\mathbf{u}_k,{\tilde{\Xi}_k}),\\
		\hat{E}_k(\mathbf{u},{\tilde{\Xi}}) \coloneqq E(\mathbf{u}_k,{\tilde{\Xi}_k})= \mathcal{G}^\dagger\scalebox{1.5}{$\epsilon$}(\mathbf{u}_k,{\tilde{\Xi}_k}),\\
		\hat{\delta}_k(\omega) \coloneqq \delta(\omega_k), \quad \hat{D}_k(\omega) \coloneqq D(\omega_k)= \mathcal{G}^\dagger\delta(\omega_k).
	\end{matrix}\label{imp_defs}
\end{equation}
Notice that \eqref{Lsys3} is an equivalent controllable linear system \cite[Lemma 1]{Alsalti2022a} to \eqref{NLsys} in transformed coordinates. Further, by the block-Brunovsky structure of the system \eqref{Lsys}, the following important property holds
\begin{subequations}
	\begin{align}
		y_{i,k+d_i} \stackrel{\eqref{Lsys}}{=} v_{i,k} &\stackrel{\eqref{basis}}{=} \phi_i(\mathbf{u}_k,\Xi_k)= g_i^\top(\hat{\Psi}_k(\mathbf{u},{{\Xi}})+\hat{E}_k(\mathbf{u},{{\Xi}}))\label{y_for_proofsA}\\
		&\stackrel{\eqref{noisy_basis}}{=}g_i^\top(\hat{\Psi}_k(\mathbf{u},{\tilde{\Xi}})+\hat{E}_k(\mathbf{u},{\tilde{\Xi}}) + \hat{D}_k(\omega)).\label{y_for_proofsB}
	\end{align}\label{y_for_proofs}
\end{subequations}
\indent Similar to \cite[Assumption 4]{Alsalti2022b}, we assume that the previously collected trajectory evolve in a compact subset of the input-state space $\Omega\subseteq\mathbb{R}^m\times\mathbb{R}^n$. This, along with local Lipschitz continuity properties of the function $\Phi(\mathbf{u}_k,\Xi_k)$ (see \eqref{definition_of_v_and_Phi} and the discussion above it), and the continuity of chosen set of basis functions, guarantee a uniform upper bound on the approximation error $\scalebox{1.5}{$\epsilon$}(\mathbf{u}_k,\Xi_k)$ for all $(\mathbf{u}_k,\Xi_k)\in\Omega$. This is summarized in the following assumption, where an explicit value of the uniform bound is assumed to be given.
\begin{assumption}\label{bounded_err_assmp}
	The error in the basis function approximation $\hat{\scalebox{1.5}{$\epsilon$}}_k(\mathbf{u},\Xi)$ is uniformly upper bounded by $\varepsilon^*>0$, i.e., $\norm*{\hat{\scalebox{1.5}{$\epsilon$}}_k(\mathbf{u},\Xi)}_\infty\leq\varepsilon^*$, for all $(\mathbf{u},\Xi)\in\Omega\subseteq\mathbb{R}^m\times\mathbb{R}^n$, where $\Omega$ is a compact subset of the input-state space.
\end{assumption}
\begin{remark}\label{remark_on_E_bnd} Since $\mathcal{G}$ has full row rank, Assumption \ref{bounded_err_assmp} implies that
	\begin{equation}
			\norm*{\hat{E}_k(\mathbf{u},\Xi)}_\infty \leq \norm*{\mathcal{G}^\dagger}_\infty\norm*{\hat{\scalebox{1.5}{$\epsilon$}}_k(\mathbf{u},\Xi)}_\infty\leq \norm*{\mathcal{G}^\dagger}_\infty\varepsilon^*.\label{tilde_E_bdd_asmp}
	\end{equation}
%	with $\mathcal{G}^\dagger = \mathcal{G}^\top\left(\mathcal{G}\mathcal{G}^\top\right)^{-1}$.
\end{remark}
\begin{remark}
	Note that $\delta(\mathbf{0})=\mathbf{0}$ and, by local Lipschitz continuity $\Phi$ and $\Psi$ as well as boundedness of $\mathbf{w}_k$ there exists a $K_w>0$ such that $\norm*{\delta(\omega_k)}_\infty\leq K_w w^*$ for all $k\geq0$.
\end{remark}
\begin{remark}
	The Lipschitz continuity constant of the basis functions $\Psi(\mathbf{u}_k,\Xi_k)$ in $\Xi$ on the compact set $\Omega$ is denoted by $K_\Psi$.
\end{remark}
\begin{remark}
	In the main result of this technical report (Theorem \ref{main_result_of_report}), it will be shown that the closed-loop trajectories of the system under application of the robust data-driven predictive controller also evolve in a compact subset of the input-state space. In particular, the set of initial states for which the predictive control scheme is initially feasible serves as the guaranteed region of attraction of the closed-loop system and, since it is ($d_{\textup{max}}$-step) invariant as shown in Theorem \ref{main_result_of_report}, the closed-loop trajectories evolve in a compact subset of the input-state space.
\end{remark}
For the nominal case when the basis function approximation error is zero and the data is noiseless, i.e., $\hat{\scalebox{1.5}{$\epsilon$}}_k(\mathbf{u},\Xi)=\hat{\delta}_k(\omega)\equiv0$, then \cite[Theorem 2]{Alsalti2022a} sates that if $\{\hat{\Psi}_k(\mathbf{u},\Xi)\}_{k=0}^{N-1}$ is persistently exciting of order $L+n$, then any sequence $\left\lbrace \mathbf{\bar{u}}_k\right\rbrace_{k=0}^{L-1},\,\{{\bar{y}}_{i,k}\}_{k=0}^{L+d_i-1}$,\, is an input-output trajectory of \eqref{NLsys} if and only if there exists an $\alpha\in\mathbb{R}^{N-L+1}$ such that the following holds
\begin{equation}
	\begin{bmatrix}
		H_L(\hat{\Psi}(\mathbf{u},\Xi)) \\ H_{L+1}(\Xi)
	\end{bmatrix}\alpha = \begin{bmatrix}
		\hat{\Psi}(\mathbf{\bar{u}},\bar{\Xi}) \\ \bar{\Xi}
	\end{bmatrix}\label{DB_rep},
\end{equation}
where $\bar{\Xi}_k = \begin{bmatrix}\bar{y}_{1,[k,k+d_1-1]}^\top & \cdots & \bar{y}_{m,[k,k+d_m-1]}^\top\end{bmatrix}^\top$. In the next section, we use \eqref{DB_rep} as a prediction model in the robust data-driven nonlinear predictive control scheme and later prove that it results in recursive feasibility and practical exponential stability of the closed-loop system. This is done using only input and noisy output data of the nonlinear system \eqref{NLsys} and assuming that the basis functions approximation error in \eqref{basis} is uniformly upper bounded as in Assumption \ref{bounded_err_assmp}.

	\section{Robust nonlinear predictive control scheme}\label{rob_scheme_sec}
Using the data-based representation \eqref{DB_rep} of all the trajectories of the nonlinear system \eqref{NLsys}, we presented in \cite[Theorem 5]{Alsalti2022b} a data-driven $d_{\textup{max}}-$step robust nonlinear predictive control scheme, where $d_{\textup{max}}$ is the maximum relative degree of all outputs. This scheme is inherently an output feedback nonlinear predictive controller and only uses previously collected input and noisy output measurements of \eqref{NLsys} which are denoted $\{\mathbf{u}_k^{\textup{d}}\}_{k=0}^{N-1},\,\{\tilde{y}_{i,k}^{\textup{d}}\}_{k=0}^{N+d_i-1}$ (and the corresponding state sequence $\{\tilde{\Xi}_k^{\textup{d}}\}_{k=0}^{N-1}$, see \eqref{definition_of_plain_Xi}). The data are collected such that $\{\hat{\Psi}_k(\mathbf{u}^{\textup{d}},\tilde{\Xi}^{\textup{d}})\}_{k=0}^{N-1}$ is persistently exciting of order $L+d_{\textup{max}}+n$ which, for sufficiently small $\varepsilon^*$ and $w^*$, implies persistency of excitation of $\{\hat{\Psi}_k(\mathbf{u}^{\textup{d}},\tilde{\Xi}^{\textup{d}}) + \hat{E}_k(\mathbf{u}^{\textup{d}},\tilde{\Xi}^{\textup{d}}) + \hat{D}_k(\omega^{\textup{d}})\}_{k=0}^{N-1}$ of the same order. This is the case when, e.g., $\norm*{H_{L+d_{\textup{max}}}(\hat{E}(\mathbf{u}^{\textup{d}},\tilde{\Xi}^{\textup{d}})+\hat{D}(\omega^{\textup{d}}))}_2<\sigma_{\textup{min}}(H_{L+d_{\textup{max}}}(\hat{\Psi}(\mathbf{u}^{\textup{d}},\tilde{\Xi}^{\textup{d}})))$. This is summarized in the following assumption.
\begin{assumption}\label{PE_asmp}
	The input-output data $(\mathbf{u}^{\textup{d}},\mathbf{\tilde{y}}^{d})$ are collected such that $\{\hat{\Psi}_k(\mathbf{u}^{\textup{d}},\tilde{\Xi}^{\textup{d}})\}_{k=0}^{N-1}$ is persistently exciting of order $L+d_{\textup{max}}+n$.
\end{assumption}
In Assumption \ref{PE_asmp}, $L$ is the prediction horizon of the proposed predictive control scheme. We require that the prediction horizon satisfies the following assumption.
\begin{assumption}\label{horizon_rob_asmp}
	The prediction horizon satisfies $L\geq d_{\textup{max}}$.
\end{assumption}
To implement the data-driven predictive control scheme, the minimization problem \eqref{DDNMPC2} below is solved at time $t$. Then, once a solution is obtained, we apply the first $d_{\textup{max}}-$instances of the optimal predicted input to the system, i.e., $\bar{\mathbf{u}}^*_{[0,d_{\textup{max}}-1]}(t)$. Finally, the prediction horizon is shifted $d_{\textup{max}}$ steps to the future and the process is repeated again at time $t+d_{\textup{max}}$. The optimal control problem at time $t$ takes the form
\begin{subequations}
	\begin{align}
		J^*_t&=\min\limits_{\tiny\begin{matrix}\mathbf{\bar{u}}(t),\bar{y}_i(t)\\\alpha(t),\sigma(t)\end{matrix}}\sum\limits_{k=0}^{L-1}\ell(\mathbf{\bar{u}}_k(t),\mathbf{\bar{y}}_k(t))+\lambda_\alpha\max\{\varepsilon^*,w^*\}\norm*{\alpha(t)}_2^2+\lambda_\sigma\norm*{\sigma(t)}_2^2\label{opt_val_fcn_def}\\[2ex]
		\text{s.t.}&\begin{bmatrix}\hat{\Psi}(\mathbf{\bar{u}}(t),\bar{\Xi}(t)) \\ \bar{\Xi}(t)\end{bmatrix}+\sigma(t) =
		\begin{bmatrix} H_{L+d_{\textup{max}}}(\hat{\Psi}(\mathbf{u}^{\textup{d}},{\tilde{\Xi}^{\textup{d}}}))  \\ H_{L+d_{\textup{max}}+1}({\tilde{\Xi}^{\textup{d}}})\end{bmatrix} \alpha(t),\label{pc2_willems}\\
		&\begin{bmatrix}
			\mathbf{\bar{u}}_{[-d_{\textup{max}},-1]}(t)\\
			\mathbf{\bar{y}}_{[-d_{\textup{max}},-1]}(t)
		\end{bmatrix} = \begin{bmatrix}
			\mathbf{{u}}_{[t-d_{\textup{max}},t-1]}\\
			{\mathbf{\tilde{y}}_{[t-d_{\textup{max}},t-1]}}
		\end{bmatrix}\label{pc2_ini},\\
		&\bar{y}_{i,[L,L+d_i-1]}(t) = \mathbf{0},\quad \forall i\in\mathbb{Z}_{[1,m]},\label{pc2_term}\\
		&\bar{\Xi}_k(t) = \begin{bmatrix}\bar{y}_{1,[k,k+d_1-1]}^\top(t) & \dots & &\bar{y}_{m,[k,k+d_m-1]}^\top(t)\end{bmatrix}^\top,\\
		&\scalebox{1}{$\norm*{\sigma_k(t)}_\infty\leq K_{\Psi}w^* + (\varepsilon^* + K_w w^*)\norm*{\mathcal{G}^\dagger}_\infty(1+\norm*{\alpha(t)}_1),\label{slack_const}$}\\
		&\mathbf{\bar{u}}_k(t)\in\mathcal{U},\quad\forall k\in\mathbb{Z}_{[-d_{\textup{max}},L-1]}.\label{rob_input_const_set}
	\end{align}%
	\label{DDNMPC2}%
\end{subequations}
The notation used in \eqref{DDNMPC2} is explained as follows (and summarized in Table \ref{notation_table}). The sequences $\{\mathbf{u}_k^{\textup{d}}\}_{k=0}^{N-1},\{y_{i,k}^{\textup{d}}\}_{k=0}^{N+d_i-1}$, for $i\in\mathbb{Z}_{[1,m]}$, and $\{\Xi_k^{\textup{d}}\}_{k=0}^{N}$ denote previously collected (noideless) input, output and state data, respectively. The sequences $\{\mathbf{\bar{u}}_k(t)\}_{k=-d_{\textup{max}}}^{L-1},\{\bar{\Xi}_k(t)\}_{k=-d_{\textup{max}}}^{L}$ and $\{{\bar{y}}_{i,k}(t)\}_{k=-d_{\textup{max}}}^{L+d_i-1}$ constitute the predicted input, state and output trajectories, predicted at time $t$, whereas $\mathbf{u}_t,\Xi_t,y_{i,t}$ denote the current (noiseless) input, state, and outputs of the system at time $t$, respectively. The optimal predicted input, output and state sequences (i.e., solutions of \eqref{DDNMPC2} at time $t$) are denoted by $\{\mathbf{\bar{u}}_k^*(t)\}_{k=-d_{\textup{max}}}^{L-1},\{\bar{\Xi}_k^*(t)\}_{k=-d_{\textup{max}}}^{L}$ and $\{{\bar{y}}_{i,k}^*(t)\}_{k=-d_{\textup{max}}}^{L+d_i-1}$. We denote the open-loop cost and the optimal open-loop cost of \eqref{DDNMPC2} by $J_t$ and $J^*_t$, respectively. In \eqref{opt_val_fcn_def}, we will consider quadratic stage cost functions that penalize the difference of the predicted input and output trajectories from a desired equilibrium that is known a priori. In particular, we consider
\begin{equation}
	\ell(\mathbf{\bar{u}}_k(t),\mathbf{\bar{y}}_k(t)) = \norm*{\mathbf{\bar{u}}_k - \mathbf{u}^s}_R^2 + \norm*{\mathbf{\bar{y}}_k - \mathbf{y}^s}_Q^2,\label{stage_cost_fcn}
\end{equation}
where $Q = Q^\top \succ0$ and $R=R^\top \succ 0$ and $\mathbf{u}^s,\,\mathbf{y}^s$ are the input and output equilibrium points that correspond to an equilibrium state $\mathbf{x}^s$ for the system in \eqref{NLsys}. In the following theoretical analysis, we consider deviation from the origin, i.e., $\mathbf{u}^s=\mathbf{y}^s=\mathbf{0}$. The results similarly apply to nonzero fixed equilibrium points subject to changes in the structures of some constants of the proofs of Theorems \ref{recrusive_feasibility_theorem} and \ref{main_result_of_report}.
\begin{remark}
	Note that $\Xi^s=T(\mathbf{x}^s)$ is an equilibrium point of system \eqref{Lsys3}. Furthermore, $\Xi^s$ is an equilibrium point of \eqref{Lsys3} if and only if $y_{i,k}=y_i^s$ for $d_i$ consecutive instances, for all $i\in\mathbb{Z}_{[1,m]}$, with $y_i^s$ being the corresponding element of $\mathbf{y}^s$.
\end{remark}
The constraint \eqref{pc2_willems} uses a modified version of \eqref{DB_rep} to generate the predicted inputs and outputs $\mathbf{\bar{u}}(t),\bar{y}_i(t)$. Notice that, due to the basis function approximation and noisy output measurements, \eqref{DB_rep} does not hold anymore. Therefore, we introduce a slack variable $\sigma(t) = [\sigma_{\Psi}^\top(t)\quad \sigma_{\Xi}^\top(t)]^\top$ and penalize its norm in the cost function (cf. \cite{Coulson20, Berberich203}). Additionally, we require the slack variable to be bounded as in \eqref{slack_const}. Furthermore, to mitigate the effect of uncertainties on the accuracy of the predictions, we penalize $\norm*{\alpha(t)}_2$ in the cost function as well. This is because larger values of $\norm*{\alpha(t)}_2$ amplify the effect of the basis function approximation and noise in the data and, hence, smaller norms are preferred. In \eqref{pc2_ini}, the state at time $t-d_{\textup{max}}$ is fixed by the past $d_{\textup{max}}-$instances of the output (see \eqref{definition_of_plain_Xi}), where $d_{\textup{max}}$ is the maximum relative degree of all outputs as well as the system's controllability index. In contrast, the past $d_{\textup{max}}-$instances of the input are used to \emph{implicitly} fix the state at time $t$, through their effect on the outputs $y_{i,[t,t+d_{\textup{max}}-1]}$ (see \eqref{y_for_proofs}). To show stability, we enforce terminal equality constraints \eqref{pc2_term} that forces the predicted state to be zero at the end of the prediction horizon.\par
The system's inputs and outputs are subject to pointwise-in-time constraints, i.e., $\mathbf{{u}}_t\in\mathcal{U}\subseteq\mathbb{R}^m$, $\mathbf{y}_t\in\mathcal{Y}\subseteq\mathbb{R}^m$, and that $(\mathbf{u}^s,\mathbf{y}^s)\in\text{int}(\mathcal{U}\times\mathcal{Y})$. In this work, we assume that $\mathcal{U}$ is compact and we do not consider output constraints i.e., $\mathcal{Y}=\mathbb{R}^m$, however, similar arguments can be made as in \cite{Berberich205} to ensure output constraint satisfaction.\par
The goal of this report is to show that the control scheme \eqref{DDNMPC2} is recursively feasible and results in practical exponential stability of the closed-loop system. Similar to \cite{Berberich203}, the proof is divided into four parts. First, in the remainder of this section, we provide in Lemma \ref{lemma_output_difference} a bound on the difference between $|\breve{y}_{i,t+k} - \bar{y}_{i,k}^*(t)|$ similar to those derived in \cite[Theorems 3 and 4]{Alsalti2022a}, where $\breve{y}_{i,t+k}$ for $k\in\mathbb{Z}_{[0,L+d_i-1]}$ is the output which would have been obtained by simulating the optimal predicted input $\bar{\mathbf{u}}^*(t)$. Second, in Section \ref{lyap_sec} we propose a Lyapunov function candidate and show that it is locally lower and upper bounded. Third, in Section \ref{rec_feas_sec} we show recursive feasibility of the predictive control scheme, and finally, in Section \ref{prac_stb_sec} we show practical stability and exponential convergence of the Lyapunov function to a region whose size depends on the basis function approximation error bound $\varepsilon^*$ and the noise bound $w^*$.\par
\begin{table}[t]
	\caption{Summary of notation used throughout the report}
	\label{notation_table}
	\begin{center}\begin{tabular}{ |p{2cm}|p{10cm}|  }
			\hline
			\textbf{Expression} &  \textbf{Meaning}\\
			\hline
			$\mathbf{u}^s$ & User-specified input equilibrium.\\
			$\mathbf{y}^s$ & User-specified output equilibrium.\\
			$\Xi^s$ & User-specified state equilibrium (see \eqref{definition_of_plain_Xi}).\\
			$\mathbf{u}^{\textup{d}}$ & Previously collected input data.\\
			$\mathbf{y}^{\textup{d}},\,\mathbf{\tilde{y}}^{\textup{d}}$ & Previously collected noiseless and noisy output data, respectively.\\
			$\Xi^{\textup{d}},\,\tilde{\Xi}^{\textup{d}}$ & Previously collected noiseless and noisy state data, respectively (see \eqref{definition_of_plain_Xi}).\\
			$\mathbf{u}_t$ & Online (closed-loop) input data at time $t$.\\
			$\mathbf{y}_t$ & Online (closed-loop) output data at time $t$.\\
			$\bar{\mathbf{u}}_k(t)$ & The $k$-th instant of the predicted input sequence, predicted at time $t$.\\
			$\bar{\mathbf{y}}_k(t)$ & The $k$-th instant of the predicted output sequence, predicted at time $t$.\\
			$\bar{\Xi}_k(t)$ & The $k-$th instant of the predicted state sequence, predicted at time $t$ (see \eqref{definition_of_plain_Xi}).\\
			$\breve{y}_{i,t+k}$ & The $k$-th instant of the $i-$th output obtained when applying the optimal predicted input sequence predicted at time $t$, i.e., $\bar{\mathbf{u}}^*(t)$ to the system (see \eqref{y_for_proofs}).\\
			$\breve{\Xi}_{t+k}$ & The $k-$th instant of the state obtained when applying the optimal predicted input $\mathbf{\bar{u}}^*(t)$ to the system at time $t$ (see \eqref{definition_of_plain_Xi}).\\
			$L$ & Prediction horizon.\\
			$\mathcal{U}$ & Input constraint set.\\
			$J^*_t$ & Optimal open-loop cost at time $t$.\\
			$J_t$ & Open-loop cost at time $t$.\\
			$\ell(\bar{\mathbf{u}}_k(t),\bar{\mathbf{y}}_k(t))$ & Stage cost function at time $t$.\\
			\hline
	\end{tabular}\end{center}
\end{table}
\begin{lemma}\label{lemma_output_difference}
Let $\bar{\mathbf{u}}^*_{[0,L-1]}(t),\bar{y}_{i,[0,L+d_i-1]}^*(t),\alpha^*(t),\sigma^*(t)$ be solutions of \eqref{DDNMPC2} at time $t$, and let $\breve{\Xi}_{[t,t+L]}$ and $\breve{y}_{i,[t,t+L+d_i-1]}$ be the state and output of system \eqref{Lsys3} resulting from applying $\bar{\mathbf{u}}^*_{[0,L-1]}(t)$ to \eqref{Lsys3} at time $t$. Then, for all $k\in\mathbb{Z}_{[0,L+d_i-1]}$,
\begin{align}
	\left|\breve{y}_{i,t+k} - \bar{y}_{i,k}^*(t)\right|&\leq\mathcal{P}^{k + d_{\textup{max}} - d_i}(K_{\Xi})\Big(\varepsilon^*(1+\norm*{\alpha^*(t)}_1) + (1 + K_w)w^*\norm*{\alpha^*(t)}_1 + (1+\norm*{\mathcal{G}}_\infty)\norm*{\sigma^*(t)}_\infty\Big)
	\label{output_difference}
\end{align}
where $K_\Xi$ is the Lipschitz continuity constant of $\Phi(\cdot,\cdot)$ w.r.t. $\Xi$, $K_w>0$ and $\mathcal{P}^k(K_\Xi)$ is a polynomial of the form $\mathcal{P}^k(K_\Xi) = (K_\Xi)^k + (K_\Xi)^{k-1} + \dots + K_\Xi +1$.
\end{lemma}
\begin{proof}
	The output obtained when applying $\bar{\mathbf{u}}_{[0,L-1]}^*(t)$ to the system are given by
	\begin{equation}
		\breve{y}_{i,t+k} \stackrel{\eqref{y_for_proofs}}{=} \phi_i(\bar{\mathbf{u}}_{k-d_i}^*(t),\breve{\Xi}_{t+k-d_i}),
	\end{equation}
	where $\bar{\mathbf{u}}_{[-d_i,-1]}^*(t)\stackrel{\eqref{pc2_ini}}{=}\mathbf{u}_{[t-d_i,t-1]}$. In contrast, the predicted outputs are given by
\begin{equation*}
	\begin{aligned}
		\hspace{0mm}\bar{y}_{i,k}^*(t) &\stackrel{\eqref{pc2_willems},\eqref{definition_of_plain_Xi}}{=}H_1({\tilde{y}_{i,[k+d_{\textup{max}},k+N-L]}^{\textup{d}}})\alpha^*(t) - \sigma_{\Xi,k}^*(t)\\
		&= H_{1}({{y}_{i,[k+d_{\textup{max}},k+N-L]}^{\textup{d}}})\alpha^*(t)+ H_{1}({w_{i,[k+d_{\textup{max}},k+N-L]}^{\textup{d}}})\alpha^*(t)- \sigma_{\Xi,k}^*(t)\\
		&\stackrel{\eqref{y_for_proofs}}{=}g_i^\top H_1\Big(\hat{\Psi}_{[k\hspace{-0.25mm}+d_{\textup{max}}\hspace{-0.25mm}-\hspace{-0.25mm}d_i,k\hspace{-0.25mm}+\hspace{-0.25mm}N\hspace{-0.25mm}-\hspace{-0.25mm}L\hspace{-0.25mm}-\hspace{-0.25mm}d_i]}(\mathbf{u}^{\textup{d}},\tilde{\Xi}^{\textup{d}})		+\hat{E}_{[k\hspace{-0.25mm}+d_{\textup{max}}\hspace{-0.25mm}-\hspace{-0.25mm}d_i,k\hspace{-0.25mm}+\hspace{-0.25mm}N\hspace{-0.25mm}-\hspace{-0.25mm}L\hspace{-0.25mm}-\hspace{-0.25mm}d_i]}(\mathbf{u}^{\textup{d}},\tilde{\Xi}^{\textup{d}})+ \hat{D}_{[k\hspace{-0.25mm}+d_{\textup{max}}\hspace{-0.25mm}-\hspace{-0.25mm}d_i,k\hspace{-0.25mm}+\hspace{-0.25mm}N\hspace{-0.25mm}-\hspace{-0.25mm}L\hspace{-0.25mm}-\hspace{-0.25mm}d_i]}(\omega^{\textup{d}})\Big)\alpha^*\hspace{-0.25mm}(t)\\
		&\quad+ H_{\hspace{-0.5mm}1}({w_{i,[k\hspace{-0.25mm}+d_{\textup{max}}\hspace{-0.25mm},k\hspace{-0.25mm}+\hspace{-0.25mm}N\hspace{-0.25mm}-\hspace{-0.25mm}L]}^{\textup{d}}})\alpha^*\hspace{-0.25mm}(t)- \sigma_{\Xi,k}^*(t)\\
		&\stackrel{\eqref{pc2_willems},\eqref{imp_defs}}{=} g_i^\top \left(\Psi(\mathbf{\bar{u}}_{k-d_i}^*(t),\bar{\Xi}_{k-d_i}^*(t)) + \sigma_{\Psi,k-d_i}^*(t)\right) + H_1(\hat{\varepsilon}_{i,[k+d_{\textup{max}}-d_i,k+N-L-d_i]}(\mathbf{u}^{\textup{d}},{{\Xi}^{\textup{d}}}))\alpha^*(t)\\
		&\quad + H_1(\hat{\delta}_{i,[k+d_{\textup{max}}-d_i,k+N-L-d_i]}(\omega^{\textup{d}}))\alpha^*(t)+ H_1({w_{i,[k+d_{\textup{max}},k+N-L]}^{\textup{d}}})\alpha^*(t)- \sigma_{\Xi,k}^*(t)\\
		&\stackrel{\eqref{noisy_basis}}{=} \phi_{i}(\mathbf{\bar{u}}_{k-d_i}^*(t),\bar{\Xi}_{k-d_i}^*(t)) - \hat{\varepsilon}_{i}(\mathbf{\bar{u}}_{k-d_i}^*(t),\bar{\Xi}_{k-d_i}^*(t)) + g_i^\top\sigma_{\Psi,k-d_i}^*(t)- \sigma_{\Xi,k}^*(t)\\
		&\quad + H_1(\hat{\varepsilon}_{i,[k+d_{\textup{max}}-d_i,k+N-L-d_i]}(\mathbf{u}^{\textup{d}},{{\Xi}^{\textup{d}}}))\alpha^*(t)+ H_1(\hat{\delta}_{i,[k+d_{\textup{max}}-d_i,k+N-L-d_i]}(\omega^{\textup{d}}))\alpha^*(t)\\
		&\quad+ H_1({w_{i,[k+d_{\textup{max}},k+N-L]}^{\textup{d}}})\alpha^*(t).
	\end{aligned}
\end{equation*}%
The error between the two outputs for $k\in\mathbb{Z}_{[0,L+d_i-1]}$ is
\begin{equation}
	\begin{aligned}
		\breve{y}_{i,t+k} - \bar{y}_{i,k}^*(t) &= {\phi_i({\mathbf{\bar{u}}}^*_{k-d_i}(t),\breve{\Xi}_{t+k-d_i})} - {\phi_{i}\left(\mathbf{\bar{u}}^*_{k-d_i}(t),\bar{\Xi}^*_{k-d_i}(t)\right)}+ \hat{\varepsilon}_{i}\left(\mathbf{\bar{u}}^*_{k-d_i}(t),\bar{\Xi}^*_{k-d_i}(t)\right) -g_i^\top\sigma_{\Psi,k-d_i}^*(t) +\sigma_{\Xi,k}^*(t) \\
		& \quad - H_1\left(\hat{\varepsilon}_{i,[k+d_{\textup{max}}-d_i,k+N-L-d_i]}(\mathbf{u}^{\textup{d}},{\tilde{\Xi}^{\textup{d}}})\right)\alpha^*(t)- H_1(\hat{\delta}_{i,[k+d_{\textup{max}}-d_i,k+N-L-d_i]}(\omega^{\textup{d}}))\alpha^*(t)\\
		& \quad- H_1({w_{i,[k+d_{\textup{max}},k+N-L]}^{\textup{d}}})\alpha^*(t).
	\end{aligned}\label{phis_dont_cancel_out}
\end{equation}
which, using Assumption \ref{bounded_err_assmp}, \eqref{slack_const} and continuity of $\phi_i(\cdot,\cdot)$, can be bounded by
\begin{align}
	\left|\breve{y}_{i,t+k} - \bar{y}_{i,k}^*(t)\right|&\leq K_{\Xi}\norm*{\breve{\Xi}_{t+k-d_i}-\bar{\Xi}_{k-d_i}^*(t)}_\infty\label{prel_lemma_part2_a} \hspace{-3mm} + \varepsilon^*(1+\norm*{\alpha^*(t)}_1) + (1 + K_w)w^*\norm*{\alpha^*(t)}_1 + (1+\norm*{\mathcal{G}}_\infty)\norm*{\sigma^*(t)}_\infty
\end{align}
Using similar induction steps as in \cite[Theorem 3]{Alsalti2022a}, it can be shown that \eqref{prel_lemma_part2_a} is equivalent to
\begin{align}
	\left|\breve{y}_{i,t+k} - \bar{y}_{i,k}^*(t)\right|&\leq\label{prel_lemma_part2}\mathcal{P}^{k + d_{\textup{max}} - d_i}(K_{\Xi})\Big(\varepsilon^*(1+\norm*{\alpha^*(t)}_1) + (1 + K_w)w^*\norm*{\alpha^*(t)}_1 + (1+\norm*{\mathcal{G}}_\infty)\norm*{\sigma^*(t)}_\infty\Big),
\end{align}
as in \eqref{output_difference} which completes the proof.
\end{proof}
\begin{remark}
	Notice that the bound in \eqref{output_difference} increases as $k$ increases. This is due to the polynomial term in $K_{\Xi}$. For that reason, when appropriate, we will use the bound on the last instant as the upper bound for the whole output difference sequence, i.e., $\norm*{\breve{y}_{i,[t,t+L+d_i-1]}-\bar{y}^*_{i,[0,L+d_i-1]}(t)}_\infty\leq\mathcal{P}^{L+d_{\textup{max}}-1}(K_{\Xi})\Big(\varepsilon^*(1+\norm*{\alpha^*(t)}_1) + (1 + K_w)w^*\norm*{\alpha^*(t)}_1 + (1+\norm*{\mathcal{G}}_\infty)\norm*{\sigma^*(t)}_\infty\Big)$. Furthermore, if $K_{\Xi}<1$, then the bound converges for increasing $L$.
\end{remark}
	\section{Lyapunov function candidate: Proof of lower and upper bounds}\label{lyap_sec}
In this section, we propose a Lyapunov function candidate and show that it is both locally lower and upper bounded. First, note that the system in \eqref{Lsys} is a controllable and observable system (since $\mathcal{A,B,C}$ are in block-Brunovksy canonical form). Hence, there exists a matrix $P=P^\top\succ 0$ such that $W_t=\norm*{\Xi_t}_P^2$ is an input-output-to-state stability (IOSS) Lyapunov function \cite{Cai08} that satisfies
\begin{equation}
	W_{t+d_{\textup{max}}} - W_t \leq -\frac{1}{2}\norm*{\Xi_{[t,t+d_{\textup{max}}-1]}}_2^2 + {c}_1\norm*{\mathbf{v}_{[t,t+d_{\textup{max}}-1]}}_2^2 + {c}_2\norm*{\mathbf{y}_{[t,t+d_{\textup{max}}-1]}}_2^2,
	\label{IOSS}
\end{equation}
for ${c}_1,{c}_2>0$. Similar to \cite[Lemma 1]{Berberich203}, we consider the following Lyapunov function candidate $V_t = J^*_t + c_3 W_t$, for some $c_3>0$. Clearly, the following lower bound holds
\begin{equation}
	\begin{aligned}
		V_t &= \underbrace{J^*_t}_{\geq 0} + c_3 \underbrace{W_t}_{\geq \lambda_{\text{min}}(P)\norm*{\Xi_t}_2^2}\\
		V_t &\geq c_3\lambda_{\text{min}}(P)\norm*{\Xi_t}_2^2.
	\end{aligned}
	\label{rob_lower_bnd}
\end{equation}
As for the upper bound, notice that
\begin{equation}
	\begin{aligned}
		V_t &= \underbrace{J^*_t}_{\text{by optimality } \leq J_t} + c_3\underbrace{\norm*{\Xi_t}_P^2}_{\leq\lambda_{\text{max}}(P)\norm*{\Xi_t}_2^2}\\
		&\leq \sum\limits_{k=0}^{L-1}\ell (\mathbf{\bar{u}}'(t),\mathbf{\bar{y}}'(t)) + \lambda_\alpha\max\{\varepsilon^*,w^*\}\norm*{\alpha'(t)}_2^2 + \lambda_\sigma\norm*{\sigma'(t)}_2^2 + c_3\lambda_{\text{max}}(P)\norm*{\Xi_t}_2^2.
	\end{aligned}
	\label{pre_V_upper_bnd}
\end{equation}
where $\mathbf{\bar{u}}'_{[-d_{\textup{max}},L-1]}(t),\bar{y}_{i,[-d_{\textup{max}},L+d_i-1]}'(t),\alpha'(t),\sigma'(t)$ are candidate solutions to \eqref{DDNMPC2} at time $t$. In the following we propose such candidate solutions and provide an expression for the upper bound on $V_t$. By the initial condition constraints \eqref{pc2_ini}, we have that
\begin{equation*}
	\begin{matrix}
		\mathbf{\bar{u}}_{[-d_{\textup{max}},-1]}'(t) = \mathbf{u}_{[t-d_{\textup{max}},t-1]}, &\qquad & 
		\bar{y}'_{i,[-d_{\textup{max}},-1]}(t) = y_{i,[t-d_{\textup{max}},t-1]}.\\
	\end{matrix}
\end{equation*}
Since $L\geq d_{\textup{max}}$ by Assumption \ref{horizon_rob_asmp} and assuming that $0\in\textup{int}(\mathcal{U})$, then by controllability of the system \eqref{Lsys}, there exists a $\delta>0$ such that for any state $\Xi_t\in\mathbb{B}_\delta\coloneqq\{\Xi_t\in\mathbb{R}^n~|~\norm*{\Xi_t}\leq\delta\}$, there exists a synthetic input $\mathbf{v}_{[t,t+L-1]}$, and hence a control input $\mathbf{u}_j\coloneqq\gamma(\Xi_j,\mathbf{v}_j)$ for $j\in\mathbb{Z}_{[t,t+L-1]}$ with $\mathbf{u}_j\in\mathcal{U}$, which brings the state $\Xi_{[t,t+L]}$ (defined by the corresponding outputs $y_{i,[t,t+L+d_i-1]}$ as in \eqref{definition_of_plain_Xi}) to zero in $L$ steps\footnote{Notice that $d_{\textup{max}}$ is the controllability index of the system in \eqref{Lsys}.}, while satisfying (for some $\Gamma_{v}>0$)
\begin{equation}\scalebox{1}{$
		\norm*{\begin{bmatrix}
				\mathbf{v}_{[t,t+L-1]} \\ \Xi_{[t,t+L]}
		\end{bmatrix}}_2^2   \leq \Gamma_{v}\norm*{\Xi_t}_2^2$}\label{ctrb_argument_t}.
\end{equation}
Furthermore, notice that the function $\gamma(\cdot,\cdot)$ is locally Lipschitz continuous\footnote{This result is a by-product of the implicit function theorem which was used to show the existence of the feedback linearizing controller $\mathbf{u}_k$ (cf. \cite{Monaco87})}. Therefore, there exists a $K_\gamma>0$ such that the following holds $\norm*{\mathbf{u}_j}_2^2\leq K_\gamma \left(\norm*{\mathbf{v}_j}_2^2 + \norm*{\Xi_j}_2^2\right)$ for all $j\in\mathbb{Z}_{[t,t+L-1]}$ in the set $\Omega$. More generally,
\begin{equation}
	\begin{aligned}
		\norm*{\mathbf{u}_{[t,t+L-1]}}_2^2 \leq K_\gamma\left(\norm*{\mathbf{v}_{[t,t+L-1]}}_2^2+\norm*{\Xi_{[t,t+L-1]}}_2^2\right) = K_\gamma \norm*{\begin{bmatrix}\mathbf{v}_{[t,t+L-1]} \\ \Xi_{[t,t+L-1]}\end{bmatrix}}_2^2 \stackrel{\eqref{ctrb_argument_t}}{\leq} K_\gamma \Gamma_v \norm*{\Xi_t}_2^2.
	\end{aligned}\label{continuity_bnd_on_u}
\end{equation}
Since such an input sequence $\mathbf{u}_{[t,t+L-1]}$ exists, we choose it to be the candidate input sequence at time $t$, i.e.,
\begin{equation}
	\mathbf{\bar{u}}_{[-d_{\textup{max}},L-1]}'(t) = \mathbf{u}_{[t-d_{\textup{max}},t+L-1]}.
	\label{input_candidate_at_t}
\end{equation}
As for the candidate output sequence at time $t$, we make the following choice
\begin{equation}
	\begin{aligned}
		&\bar{y}'_{i,[-d_{\textup{max}},-1]}(t) \stackrel{\eqref{pc2_ini}}{=} \tilde{y}_{i,[t-d_{\textup{max}},t-1]} = {y}_{i,[t-d_{\textup{max}},t-1]} + {w}_{i,[t-d_{\textup{max}},t-1]},\\
		&\bar{y}'_{i,[0,L-1]}(t) = v_{i,[t-d_i,t+L-d_i-1]} \stackrel{\eqref{y_for_proofsA}}{=} y_{i,[t,t+L-1]},\\
		&\bar{y}'_{i,[L,L+d_i-1]}(t) \stackrel{\eqref{pc2_term}}{=} \mathbf{0}.
	\end{aligned}
	\label{output_candidate_at_t}
\end{equation}
Next, we choose a candidate solution for $\alpha'(t),\,\sigma'(t)$. Recall that $\{\hat{\Psi}_k(\mathbf{u}^{\textup{d}},\tilde{\Xi}^{\textup{d}})\}_{k=0}^{N-1}$ is persistently exciting as in Assumption~\ref{PE_asmp}. For sufficiently small $\varepsilon^*$ and $w^*$ this implies persistency of excitation of $\{\hat{\Psi}_k(\mathbf{u}^{\textup{d}},\tilde{\Xi}^{\textup{d}})+\hat{E}_k(\mathbf{u}^{\textup{d}},\tilde{\Xi}^{\textup{d}})+\hat{D}_k(\omega^{\textup{d}})\}_{k=0}^{N-1}$ of the same order (see \cite[Lemma 2]{Alsalti2022a}). The latter sequence represents the input to the LTI system in transformed coordinates \eqref{Lsys3}. Therefore, by persistency of excitation, the fundamental lemma \cite{Willems05} gives a data-based representation of all the trajectories of system \eqref{Lsys3}. In particular, there exists $\alpha_t\in\mathbb{R}^{N-L-d_{\textup{max}}+1}$ such that
\begin{equation}
	\begin{bmatrix}
		H_{L+d_{\textup{max}}}(\hat{\Psi}(\mathbf{u}^{\textup{d}},\tilde{\Xi}^{\textup{d}}) + \hat{E}(\mathbf{u}^{\textup{d}},\tilde{\Xi}^{\textup{d}}) + \hat{D}(\omega^{\textup{d}}) ) \\ H_1(\Xi_{[0,N-L-d_{\textup{max}}]}^{\textup{d}})
	\end{bmatrix} \alpha_t = \begin{bmatrix}
	\hat{\Psi}(\mathbf{u},{\Xi}) + \hat{E}(\mathbf{{u}},{\Xi}) + \hat{D}({\omega})\\ {\Xi}_{t-d_{\textup{max}}}
\end{bmatrix},\label{alpha_t_exists}
\end{equation}
where $\mathbf{u} = \mathbf{u}_{[t-d_{\textup{max}},t+L-1]},\,\Xi=\Xi_{[t-d_{\textup{max}},t+L-1]},\,\omega = \begin{bmatrix}\omega_{[t-d_{\textup{max}},t-1]}^\top & \mathbf{0}\end{bmatrix}^\top$ and
\begin{equation*}
	\Xi_{t-d_{\textup{max}}} \coloneqq \begin{bmatrix}y_{1,[t-d_{\textup{max}},t-d_{\textup{max}}+d_1-1]}^\top & \dots & y_{m,[t-d_{\textup{max}},t-d_{\textup{max}}+d_m-1]}^\top\end{bmatrix}^\top,
\end{equation*}
The terms $\hat{D}(\omega^{\textup{d}})$ and $\hat{D}({\omega})$ in \eqref{alpha_t_exists} are defined in a similar way as in \eqref{y_for_proofsB}. For the candidate solution of $\alpha'(t)$, we choose it to be the one satisfying \eqref{alpha_t_exists}, i.e.,
\begin{equation}
	\alpha'(t) = \begin{bmatrix}
		H_{L+d_{\textup{max}}}(\hat{\Psi}(\mathbf{u}^{\textup{d}},\tilde{\Xi}^{\textup{d}}) + \hat{E}(\mathbf{u}^{\textup{d}},\tilde{\Xi}^{\textup{d}}) + \hat{D}(\omega^{\textup{d}}) ) \\ H_1(\Xi_{[0,N-L-d_{\textup{max}}]}^{\textup{d}})
	\end{bmatrix}^{\dagger}\begin{bmatrix}
		\hat{\Psi}(\mathbf{\bar{u}}'(t),\Xi) + \hat{E}(\mathbf{\bar{u}}'(t),\Xi) + \hat{D}({\omega})\\ {\Xi}_{t-d_{\textup{max}}}
	\end{bmatrix},\label{candidate_alpha_t}
\end{equation}
which also satisfies (by linearity of System \eqref{Lsys3})
\begin{equation}
	\begin{bmatrix}
		H_{L+d_{\textup{max}}}(\hat{\Psi}(\mathbf{u}^{\textup{d}},\tilde{\Xi}^{\textup{d}}) + \hat{E}(\mathbf{u}^{\textup{d}},\tilde{\Xi}^{\textup{d}}) + \hat{D}(\omega^{\textup{d}}) ) \\ H_{L+d_{\textup{max}}+1}(\Xi^{\textup{d}})
	\end{bmatrix}\alpha'(t) = \begin{bmatrix}
		\hat{\Psi}(\mathbf{\bar{u}}'(t),{\Xi}) + \hat{E}(\mathbf{\bar{u}}'(t),{\Xi}) + \hat{D}({\omega})\\ \begin{bmatrix}\Xi_{[t-d_{\textup{max}},t+L-1]}^\top&\mathbf{0}\end{bmatrix}^\top
	\end{bmatrix},\label{candidate_alpha_t_satisfied_whole_sequence}
\end{equation}
The choice of a candidate solution for $\sigma'(t)$ is now obtained by plugging \eqref{candidate_alpha_t} into \eqref{pc2_willems}. This is given by
\begin{equation}
	\sigma'(t) = \begin{bmatrix}
		\sigma'_{\Psi}(t) \\ \sigma'_{\Xi}(t)
	\end{bmatrix} = \begin{bmatrix}
	\hat{\Psi}(\mathbf{\bar{u}}'(t),\Xi) - \hat{\Psi}(\mathbf{\bar{u}}'(t),\bar{\Xi}'(t)) + 	\hat{E}(\mathbf{\bar{u}}'(t),{\Xi}) + \hat{D}({\omega}) - H_{L+d_{\textup{max}}}( \hat{E}(\mathbf{u}^{\textup{d}},\tilde{\Xi}^{\textup{d}}) + \hat{D}(\omega^{\textup{d}}) )\alpha'(t) \\
	H_{L+d_{\textup{max}}+1}(\omega^{\textup{d}})\alpha'(t) - \begin{bmatrix}
		\omega_{[t-d_{\textup{max}}, t-1]}\\ \mathbf{0}\\ \mathbf{0}
	\end{bmatrix}
\end{bmatrix},
	\label{candidate_sigma_t}
\end{equation}
where, by \eqref{input_candidate_at_t}, we have $\mathbf{u}_{[t-d_{\textup{max}},t+L-1]} = \mathbf{\bar{u}}'_{[-d_{\textup{max}},L-1]}(t)$, whereas
\begin{equation*}
	\bar{\Xi}'_{[-d_{\textup{max}},L]}(t) = \begin{bmatrix}
		\Xi_{[t-d_{\textup{max}}, t-1]} + \omega_{[t-d_{\textup{max}},t-1]} \\
		\Xi_{[t,t+L-1]}\\
		\textbf{0}
	\end{bmatrix} = \begin{bmatrix}
		\Xi_{[t-d_{\textup{max}}, t-1]}\\
		\Xi_{[t,t+L-1]}\\
		\textbf{0}
	\end{bmatrix} + \begin{bmatrix}
		\omega_{[t-d_{\textup{max}},t-1]}\\
		\mathbf{0}\\ \mathbf{0}
	\end{bmatrix}.
\end{equation*}
Now we can start deriving the desired upper bounds for the terms in \eqref{pre_V_upper_bnd}. Recall that we are considering stage cost functions that penalize the deviation from $\mathbf{u}^s=\mathbf{y}^s=\mathbf{0}$. Therefore, the summation term in \eqref{pre_V_upper_bnd} can be expressed as
\begin{equation}
	\begin{aligned}
		\sum\limits_{k=0}^{L-1}\ell(\mathbf{\bar{u}}'_k(t),\mathbf{\bar{y}}'_k(t)) &= \sum\limits_{k=0}^{L-1}\left(\norm*{\mathbf{\bar{u}}'_k(t)}_R^2+\norm*{\mathbf{\bar{y}}'_k(t)}_Q^2\right) \leq \lambda_{\max}(Q,R)\norm*{\begin{bmatrix}\mathbf{\bar{u}}'_{[0,L-1]}(t)\\ \mathbf{\bar{y}}'_{[0,L-1]}(t)\end{bmatrix}}_2^2\\
		& = \lambda_{\textup{max}}(Q,R)\left(\norm*{\mathbf{\bar{u}}'_{[0,L-1]}(t)}_2^2+\norm*{\mathbf{\bar{y}}'_{[0,L-1]}(t)}_2^2\right)\\
		&\stackrel{\eqref{continuity_bnd_on_u},\eqref{output_candidate_at_t}}{\leq} \lambda_{\max}(Q,R)\left(K_\gamma \norm*{\begin{bmatrix}\mathbf{v}_{[t,t+L-1]}\\ \Xi_{[t,t+L-1]}\end{bmatrix}}_2^2 + \norm*{\begin{bmatrix}{v}_{1,[t-d_{1},t+L-d_1-1]} \\ \vdots \\ {v}_{m,[t-d_{m},t+L-d_m-1]}\end{bmatrix}}_2^2\right)\\
		&\stackrel{\eqref{continuity_bnd_on_u}}{\leq} \lambda_{\max}(Q,R)\left(K_\gamma\Gamma_v\norm*{\Xi_t}_2^2 + \norm*{\begin{bmatrix}{v}_{1,[t-d_{1},t-1]} \\ \vdots \\ {v}_{m,[t-d_{m},t-1]}\end{bmatrix}}_2^2 + \norm*{\begin{bmatrix}{v}_{1,[t,t+L-d_1-1]} \\ \vdots \\ {v}_{m,[t,t+L-d_m-1]}\end{bmatrix}}_2^2\right)\\
		&\stackrel{\eqref{y_for_proofsA}}{=} \lambda_{\max}(Q,R)\left(K_\gamma\Gamma_{v}\norm*{\Xi_t}_2^2 + \norm*{\begin{bmatrix}{y}_{1,[t,t+d_1-1]} \\ \vdots \\ {y}_{m,[t,t+d_m-1]}\end{bmatrix}}_2^2 + \norm*{\begin{bmatrix}{v}_{1,[t,t+L-d_1-1]} \\ \vdots \\ {v}_{m,[t,t+L-d_m-1]}\end{bmatrix}}_2^2\right)\\
		&= \lambda_{\max}(Q,R)\left(K_\gamma\Gamma_{v}\norm*{\Xi_t}_2^2 + \norm*{\Xi_t}_2^2 + \norm*{\mathbf{v}_{[t,t+L-1]}}_2^2\right)\\
		&\stackrel{\eqref{ctrb_argument_t}}{\leq} \lambda_{\max}(Q,R)\left(K_\gamma\Gamma_{v}\norm*{\Xi_t}_2^2 + \norm*{\Xi_t}_2^2 + \Gamma_v\norm*{\Xi_t}_2^2\right)\\
		&= \lambda_{\max}(Q,R)((K_\gamma+1)\Gamma_v+1)\norm*{\Xi_t}_2^2.
	\end{aligned}\label{ub_sum_t}
\end{equation}
 For an upper bound on $\norm*{\alpha'(t)}_2^2$, we use \eqref{candidate_alpha_t} as follows%candidate_sigma_t
	\begin{align}
		\alpha'(t) &= \begin{bmatrix}
			H_{L+d_{\textup{max}}}(\hat{\Psi}(\mathbf{u}^{\textup{d}},\tilde{\Xi}^{\textup{d}}) + \hat{E}(\mathbf{u}^{\textup{d}},\tilde{\Xi}^{\textup{d}}) + \hat{D}(\omega^{\textup{d}}) ) \\ H_1(\Xi_{[0,N-L-d_{\textup{max}}]}^{\textup{d}})
		\end{bmatrix}^{\dagger}\begin{bmatrix}
			\hat{\Psi}(\mathbf{\bar{u}}'(t),\Xi) + \hat{E}(\mathbf{\bar{u}}'(t),\Xi) + \hat{D}({\omega})\\ {\Xi}_{t-d_{\textup{max}}}
		\end{bmatrix},\label{prepre_ub_candidate_alpha_t_a}\\
		\norm*{\alpha'(t)}_2^2 &\leq \underbrace{\norm*{\begin{bmatrix}
					H_{L+d_{\textup{max}}}(\hat{\Psi}(\mathbf{u}^{\textup{d}},\tilde{\Xi}^{\textup{d}}) + \hat{E}(\mathbf{u}^{\textup{d}},\tilde{\Xi}^{\textup{d}}) + \hat{D}(\omega^{\textup{d}}) ) \\ H_1(\Xi_{[0,N-L-d_{\textup{max}}]}^{\textup{d}})
				\end{bmatrix}^{\dagger}}_2^2}_{\coloneqq c_{\textup{pe}}} \norm*{\begin{bmatrix}
			\hat{\Psi}(\mathbf{\bar{u}}'(t),\Xi) + \hat{E}(\mathbf{\bar{u}}'(t),\Xi) + \hat{D}({\omega})\\ {\Xi}_{t-d_{\textup{max}}}
		\end{bmatrix}}_2^2.\label{prepre_ub_candidate_alpha_t_b}
	\end{align}
As we will see later in the proof, the constant $c_{\textup{pe}}$ is desired to be small. In \cite{Berberich203}, it was observed that large magnitudes of $\mathbf{u}^{\textup{d}}$ correspond to smaller values of $c_{\textup{pe}}$. Such observation cannot, in general, be made in the current nonlinear setting since this depends on the basis functions, which in turn depend not only on the input but on the state as well. However, such behavior was observed for some choice of $\Psi$, e.g., monomial functions in the input ($\mathbf{u}, \mathbf{u}^2, ...$) where there is a direct relationship between $\Psi$ and $\mathbf{u}$. To obtain a bound on $\norm*{\alpha'(t)}_2^2$, we notice that the rightmost term on the RHS of \eqref{prepre_ub_candidate_alpha_t_a} can be upper bounded as follows
\begin{equation}
	\begin{aligned}
		\norm*{\begin{bmatrix}
				\hat{\Psi}(\mathbf{\bar{u}}'(t),\Xi) + \hat{E}(\mathbf{\bar{u}}'(t),\Xi) + \hat{D}({\omega})\\ {\Xi}_{t-d_{\textup{max}}}
		\end{bmatrix}}_2^2 &= \norm*{
		\hat{\Psi}(\mathbf{\bar{u}}'(t),\Xi) + \hat{E}(\mathbf{\bar{u}}'(t),\Xi) + \hat{D}({\omega})}_2^2 + \norm*{{\Xi}_{t-d_{\textup{max}}}}_2^2\\
	&\leq 2\norm*{\hat{\Psi}(\mathbf{\bar{u}}'(t),\Xi)}_2^2 + 4\norm*{\hat{E}(\mathbf{\bar{u}}'(t),\Xi)}_2^2 + 4\norm*{\hat{D}({\omega})}_2^2 + \norm*{{\Xi}_{t-d_{\textup{max}}}}_2^2
	\end{aligned}\label{bounding_the_vector_on_RHS_of_alpha}
\end{equation}
where the inequality was obtained by applying the triangular inequality to the first term along with $(a+b+c)^2 \leq 2a^2 + 4b^2 + 4c^2$ for $a,b,c\in\mathbb{R}$.
Notice that the second and third terms are bounded by Assumption \ref{bounded_err_assmp} and boundedness of the noise, respectively, while the first term can be bounded as follows (by Lipschitz continuity of the basis functions)
\begin{equation}
	\begin{aligned}
		\norm*{\hat{\Psi}(\mathbf{\bar{u}}'(t),\Xi)}_2 &= \norm*{\hat{\Psi}(\mathbf{\bar{u}}'(t),\Xi) - \hat{\Psi}(\mathbf{\bar{u}}'(t),\mathbf{0}) + \hat{\Psi}(\mathbf{\bar{u}}'(t),\mathbf{0})}_2\\
		&\leq \norm*{\hat{\Psi}(\mathbf{\bar{u}}'(t),\Xi) - \hat{\Psi}(\mathbf{\bar{u}}'(t),\mathbf{0})}_2 + \norm*{\hat{\Psi}(\mathbf{\bar{u}}'(t),\mathbf{0})}_2\\
		&\leq K_\Psi \norm*{\Xi}_2 + \norm*{\hat{\Psi}(\mathbf{\bar{u}}'(t),\mathbf{0})}_2,
	\end{aligned}\label{bound_on_basis_functions_evaluated_at_candidate_input_at_time_t}
\end{equation}
where $K_\Psi$ is the Lipschitz constant of the basis functions w.r.t. the second argument in the compact subset $\Omega$ and $\Xi=\Xi_{[t-d_{\textup{max}},t+L-1]}$ as denoted earlier. Notice that the term $\mu\coloneqq\norm*{\hat{\Psi}(\mathbf{\bar{u}}'(t),\mathbf{0})}_2$ is uniformly bounded since it corresponds to evaluating the basis functions at the candidate inputs (which belong to the compact set $\mathcal{U}$ by \eqref{rob_input_const_set}). The bound \eqref{bound_on_basis_functions_evaluated_at_candidate_input_at_time_t} can be written as
\begin{equation}
	\begin{aligned}
		\norm*{\hat{\Psi}(\mathbf{\bar{u}}'(t),\Xi)}_2 &\leq K_\Psi \norm*{\Xi_{[t-d_{\textup{max}},t+L-1]}}_2 + \mu \\
		&= K_\Psi \left(\norm*{\Xi_{[t-d_{\textup{max}},t-1]}}_2 + \norm*{\Xi_{[t,t+L-1]}}_2\right) + \mu.
	\end{aligned}\label{bound_on_basis_functions_for_alpha}
\end{equation}
By linearity and observability of System \eqref{Lsys}, there exists a $\Gamma_\Xi>0$ such that
\begin{align}
	\norm*{\begin{bmatrix}
			\mathbf{v}_{[t-d_{\textup{max}},t-1]} \\ \Xi_{[t-d_{\textup{max}},t-1]}
	\end{bmatrix}}_2^2 &\leq\Gamma_\Xi\norm*{\Xi_t}_2^2.\label{upsilon_norm_1}
\end{align}
Hence, by \eqref{upsilon_norm_1} and the controllability arguments made in \eqref{ctrb_argument_t}, one can rewrite \eqref{bound_on_basis_functions_for_alpha} as
\begin{equation}
	\begin{aligned}
		\norm*{\hat{\Psi}(\mathbf{\bar{u}}'(t),\Xi)}_2 &\leq K_\Psi \left(\sqrt{\Gamma_\Xi} + \sqrt{\Gamma_v}\right)\norm*{\Xi_t}_2 + \mu.
	\end{aligned}
\end{equation}
Plugging these bounds back into \eqref{bounding_the_vector_on_RHS_of_alpha}, and using \eqref{tilde_E_bdd_asmp} and boundedness of the noise, we obtain
\begin{equation}
	\begin{aligned}
		&\norm*{\begin{bmatrix}
				\hat{\Psi}(\mathbf{\bar{u}}'(t),\Xi) + \hat{E}(\mathbf{\bar{u}}'(t),\Xi) + \hat{D}({\omega})\\ {\Xi}_{t-d_{\textup{max}}}
		\end{bmatrix}}_2^2 \\
	&\leq 2\left(K_\Psi \left(\sqrt{\Gamma_\Xi} + \sqrt{\Gamma_v}\right)\norm*{\Xi_t}_2 + \mu\right)^2 + 4r(L + d_{\textup{max}})\norm*{\mathcal{G}^\dagger}_\infty^2((\varepsilon^*)^2+(K_w w^*)^2) + \Gamma_\Xi\norm*{{\Xi}_{t}}_2^2\\
	&\leq 4\mu^2 + 4r(L + d_{\textup{max}})\norm*{\mathcal{G}^\dagger}_\infty^2((\varepsilon^*)^2+(K_w w^*)^2) + \left(8K_\Psi(\Gamma_\Xi+\Gamma_v)+\Gamma_\Xi\right)\norm*{\Xi_t}_2^2,
	\end{aligned}\label{how_we_now_bound_RHS_vector}
\end{equation}
where the last inequality was obtained by making use of $(a+b)^2 \leq 2a^2 + 2b^2$ for $a,b\in\mathbb{R}$. Plugging this into \eqref{prepre_ub_candidate_alpha_t_b}, we obtain an upper bound on $\norm*{\alpha'(t)}_2^2$ as
\begin{equation}
	\begin{aligned}
		\norm*{\alpha'(t)}_2^2 &\leq c_{\textup{pe}} \mu_1 + c_{\textup{pe}} \mu_2 \norm*{\Xi_t}_2^2,
	\end{aligned}\label{ub_alpha_t}
\end{equation}
where $\mu_1 \coloneqq 4\mu^2 + 4r(L + d_{\textup{max}})\norm*{\mathcal{G}^\dagger}_\infty^2((\varepsilon^*)^2+(K_w w^*)^2)$ and $\mu_2 \coloneqq 8K_\Psi(\Gamma_\Xi+\Gamma_v)+\Gamma_\Xi$.
Similarly, we use the definition of $\sigma'(t)$ in \eqref{candidate_sigma_t} to provide an upper bound on its norm. First, recall that
\begin{equation}
	\sigma'_{\Psi}(t) = \hat{\Psi}(\mathbf{\bar{u}}'(t),\Xi) - \hat{\Psi}(\mathbf{\bar{u}}'(t),\bar{\Xi}'(t)) + 	\hat{E}(\mathbf{\bar{u}}'(t),{\Xi}) + \hat{D}({\omega}) - H_{L+d_{\textup{max}}}( \hat{E}(\mathbf{u}^{\textup{d}},\tilde{\Xi}^{\textup{d}}) + \hat{D}(\omega^{\textup{d}}) )\alpha'(t),
\end{equation}
which, by Assumption \ref{bounded_err_assmp} as well as the boundedness of the output noise and local Lipschitz continuity of the basis functions, can be upper bounded by
\begin{equation}
	\begin{aligned}
		\norm*{\sigma'_{\Psi}(t)}_\infty &\leq \norm*{\hat{\Psi}(\mathbf{\bar{u}}'(t),\Xi) - \hat{\Psi}(\mathbf{\bar{u}}'(t),\bar{\Xi}'(t))}_{\infty} + \norm*{\hat{E}(\mathbf{\bar{u}}'(t),{\Xi})}_\infty + \norm*{\hat{D}({\omega})}_\infty + \norm*{H_{L+d_{\textup{max}}}( \hat{E}(\mathbf{u}^{\textup{d}},\tilde{\Xi}^{\textup{d}}) + \hat{D}(\omega^{\textup{d}}) )\alpha'(t)}_\infty\\
		&\leq K_{\Psi}\norm*{\Xi - \bar{\Xi}'(t)}_{\infty} + \norm*{\mathcal{G}^\dagger}_\infty\varepsilon^* + \norm*{\mathcal{G}^\dagger}_\infty K_w w^* + (\norm*{\mathcal{G}^\dagger}_\infty\varepsilon^* + \norm*{\mathcal{G}^\dagger}_\infty K_w w^*)\norm*{\alpha'(t)}_1\\
		&\leq K_{\Psi}w^* + (\varepsilon^* + K_w w^*)\norm*{\mathcal{G}^\dagger}_\infty(1+\norm*{\alpha'(t)}_1),
	\end{aligned}\label{bound_on_sigma_psi}
\end{equation}
for $K_\Psi>0$. Similarly, for $\sigma_{\Xi}'(t)$ we have
\begin{equation}
	\norm*{\sigma_{\Xi}'(t)}_\infty \leq w^*(1+\norm*{\alpha'(t)}_1).\label{bound_on_sigma_xi}
\end{equation}
Therefore, we have the following bound
\begin{equation}
	\begin{aligned}
		\norm*{\sigma'(t)}_{\infty} &= \max\{\norm*{\sigma_\Psi'(t)}_{\infty} , \norm*{\sigma_\Xi'(t)}_{\infty} \}\\
		&\leq K_{\Psi}w^* + (\varepsilon^* + K_w w^*)\norm*{\mathcal{G}^\dagger}_\infty(1+\norm*{\alpha'(t)}_1).
	\end{aligned}\label{candidate_sigma_at_t_satisfies_slack_const}
\end{equation}
Clearly, this implies that \eqref{slack_const} is satisfied. To bound $V_t$ in \eqref{pre_V_upper_bnd}, we need an upper bound on $\norm*{\sigma'(t)}_2^2$. For that we write
\begin{equation}
	\begin{aligned}
		\norm*{\sigma'(t)}_2^2 &\leq \left((r+n)(L+d_{\textup{max}}) + n\right) \norm*{\sigma'(t)}_\infty^2\\
		&\leq \left((r+n)(L+d_{\textup{max}}) + n\right) \left(K_{\Psi}w^* + (\varepsilon^* + K_w w^*)\norm*{\mathcal{G}^\dagger}_\infty(1+\norm*{\alpha'(t)}_1)\right)^2\\
		&\leq \left((r+n)(L+d_{\textup{max}}) + n\right)\left(2 (K_{\Psi}w^*)^2 + 2(\varepsilon^* + K_w w^*)^2\norm*{\mathcal{G}^\dagger}_\infty^2(2+2\norm*{\alpha'(t)}_1^2)\right)\\
		&= 	2\left((r+n)(L+d_{\textup{max}}) + n\right)\left((K_{\Psi}w^*)^2 + 2(\varepsilon^* + K_w w^*)^2\norm*{\mathcal{G}^\dagger}_\infty^2(1+\norm*{\alpha'(t)}_1^2)\right)
	\end{aligned}
\label{ub_sigma_t}
\end{equation}
where an upper bound on $\norm*{\alpha'(t)}_2^2$ can be obtained from \eqref{ub_alpha_t}. Finally, we substitute \eqref{ub_sum_t},\eqref{ub_alpha_t},\eqref{ub_sigma_t} back into \eqref{pre_V_upper_bnd} to obtain the following result
\begin{equation}
	\begin{aligned}
		V_t \leq c_4 \norm*{\Xi_t}_2^2 + c_5
	\end{aligned}\label{ub_V_t}
\end{equation}
where
\begin{equation*}
	\begin{aligned}
		c_4 &\coloneqq \lambda_{\max}(Q,R)((K_\gamma+1)\Gamma_v+1) +c_3\lambda_{\textup{max}}(P)\\
		&+ \left(\lambda_\alpha\max\{\varepsilon^*,w^*\} + 2\lambda_\sigma\left((r+n)(L+d_{\textup{max}}) + n\right)\left((K_{\Psi}w^*)^2 + 2(\varepsilon^* + K_w w^*)^2\norm*{\mathcal{G}^\dagger}_\infty^2\right)\right) c_{\textup{pe}} \mu_2,\\
		c_5 &\coloneqq 2\lambda_\sigma\left((r+n)(L+d_{\textup{max}}) + n\right)\left((K_{\Psi}w^*)^2 + 2(\varepsilon^* + K_w w^*)^2\norm*{\mathcal{G}^\dagger}_\infty^2\right) + \lambda_\alpha\max\{\varepsilon^*,w^*\}c_{\textup{pe}}\mu_1.
	\end{aligned}
\end{equation*}
Notice that $c_5$ has the property that it goes to zero as $\max\{\varepsilon^*,w^*\}\to0$.
	\section{Recursive feasibility}\label{rec_feas_sec}
In this section, we show that if the robust data-driven nonlinear predictive control scheme \eqref{DDNMPC2} is feasible at time $t$, then it is feasible at time $t+d_{\textup{max}}$, given that $\varepsilon^*$ and $w^*$ are small enough. Theorem \ref{recrusive_feasibility_theorem} below follows similar ideas as in \cite[Proposition 1]{Berberich203}. Showing recursive feasibility for all time $t\in\mathbb{N}$ will follow in Section \ref{prac_stb_sec} when it is shown that the sub-level set defined by $V_{ROA}$ is invariant.
\begin{theorem}\label{recrusive_feasibility_theorem}
	For any $V_{ROA}>0$, there exists an $\bar{\varepsilon}>0$ and $\bar{w}>0$ such that for all $\varepsilon^*\leq\bar{\varepsilon}$ and $w^*\leq\bar{w}$, if $V_t\leq V_{ROA}$ for $t\geq0$, then the optimization problem \eqref{DDNMPC2} is feasible at time $t+d_{\textup{max}}$.
\end{theorem}
\begin{proof}
	Since $V_t\leq V_{ROA}$, then $J^*_t\leq V_{ROA}$ and, hence, the optimization problem \eqref{DDNMPC2} is feasible at time $t$. Let $\mathbf{\bar{u}}^*(t),\mathbf{\bar{y}}^*(t),\alpha^*(t),\sigma^*(t)$ be the optimal solutions of \eqref{DDNMPC2} at time $t$. We propose the following candidate solution for the input sequence at time $t+d_{\textup{max}}$ for a $d_{\textup{max}}-$step predictive control scheme
	\begin{equation}
		\begin{aligned}
			\mathbf{\bar{u}}'_{[-d_{\textup{max}},-1]}(t+d_{\textup{max}}) &= \mathbf{\bar{u}}^*_{[0,d_{\textup{max}}-1]}(t) = \mathbf{{u}}_{[t,t+d_{\textup{max}}-1]},\\
			\mathbf{\bar{u}}'_{[0,L-d_{\textup{max}}-1]}(t+d_{\textup{max}}) &= \mathbf{\bar{u}}^*_{[d_{\textup{max}},L-1]}(t),
		\end{aligned}\label{pre_u_candidate_tdmax}
	\end{equation}
	where the instances $[-d_{\textup{max}},-1]$ are fixed by the initial conditions \eqref{pc2_ini} at $t+d_{\textup{max}}$, while the instances from $[0,L-d_{\textup{max}}-1]$ are specified as the shifted, previously optimal, input sequence.
	As for the candidate output sequence, we use the output which would have been obtained by applying $\mathbf{\bar{u}}_{[0,L-1]}^*(t)$ at time $t$ to the system, i.e., $\breve{y}_{i,[t,t+L+d_i-1]}$, as follows
	\begin{equation}
		\begin{aligned}
			\bar{y}'_{i,[-d_{\textup{max}},-1]}(t+d_{\textup{max}}) &\stackrel{\eqref{pc2_ini}}{=} \tilde{y}_{i,[t,t+d_{\textup{max}}-1]} = \underbrace{y_{i,[t,t+d_{\textup{max}}-1]}}_{=\breve{y}_{i,[t,t+d_{\textup{max}}-1]}} + w_{i,[t,t+d_{\textup{max}}-1]},\\
			\bar{y}'_{i,[0,L+d_i-d_{\textup{max}}-1]}(t+d_{\textup{max}}) &= \breve{y}_{i,[t+d_{\textup{max}},t+L+d_{i}-1]}.
		\end{aligned}\label{pre_y_candidate_tdmax}
	\end{equation}
	Recall that $\bar{y}_{i,[L,L+d_i-1]}^*(t)=0$ by the terminal constraints of the previous iteration (i.e., at time $t$). Therefore, the following inequality holds by Lemma \ref{lemma_output_difference}
	\begin{equation}
		\begin{aligned}
			\left|\breve{y}_{i,t+k}\right|&\leq\mathcal{P}^{k + d_{\textup{max}} - d_i}(K_{\Xi})\Big(\varepsilon^*(1+\norm*{\alpha^*(t)}_1) + (1 + K_w)w^*\norm*{\alpha^*(t)}_1 + (1+\norm*{\mathcal{G}}_\infty)\norm*{\sigma^*(t)}_\infty\Big), \quad \textup{for }k\in\mathbb{Z}_{[L,L+d_i-1]}.
		\end{aligned}
		\label{bound_on_y_to_show_exitence_of_u}
	\end{equation}
	As $\max\{\varepsilon^*, w^*\}\to0$, the RHS of \eqref{bound_on_y_to_show_exitence_of_u} goes to zero as well since $\norm*{\sigma^*(t)}_\infty$ is bounded by \eqref{slack_const} while $\norm*{\alpha^*(t)}_1$ remains bounded by $\lambda_\alpha\max\{\varepsilon^*, w^*\}\norm*{\alpha^*(t)}_2^2 \leq J^*(t) \leq V_{ROA}$, i.e.,
	\begin{equation*}
			\norm*{\alpha^*(t)}_1 \leq \sqrt{N-L-d_{\textup{max}}+1}\norm*{\alpha^*(t)}_2 \leq \sqrt{N-L-d_{\textup{max}}+1}\sqrt{\frac{V_{ROA}}{\lambda_\alpha\max\{\varepsilon^*, w^*\}}}.
	\end{equation*}
	This implies that the corresponding state $\breve{\Xi}_{t+L}=\begin{bmatrix}
		\breve{y}_{1,[t+L,t+L+d_1-1]}^\top & \dots & \breve{y}_{m,[t+L,t+L+d_m-1]}^\top
	\end{bmatrix}^\top$ approaches zero as the bounds $\max\{\varepsilon^*, w^*\}\to0$. Denote by $\bar{\Xi}'_k(t+d_{\textup{max}})$ the $k-$th instant of the state candidate solution predicted at time $t+d_{\textup{max}}$, and notice that the choice of candidate outputs $\bar{y}'_{i,[L-d_{\textup{max}},L+d_i-d_{\textup{max}}-1]}(t+d_{\textup{max}}) = \breve{y}_{i,[t+L,t+L+d_i-1]}$ implies that $\bar{\Xi}'_{L-d_{\textup{max}}}(t+d_{\textup{max}}) = \breve{\Xi}_{t+L}$, and hence $\bar{\Xi}'_{L-d_{\textup{max}}}(t+d_{\textup{max}})$ also approaches zero as $\max\{\varepsilon^*, w^*\}\to0$. Since System \eqref{Lsys} is controllable, there exists a synthetic input $\mathbf{v}_{[t+L,t+L+d_{\textup{max}}-1]}$, and hence $\mathbf{\bar{u}}'_{[L-d_{\textup{max}},L-1]}(t+d_{\textup{max}})$, e.g., a deadbeat controller, that brings the state $\bar{\Xi}'_{[L-d_{\textup{max}},L-1]}(t+d_{\textup{max}})$ and the corresponding outputs $\bar{y}'_{i,[L-d_{\textup{max}},L+d_i-1]}(t+d_{\textup{max}})$ to zero in $d_{\textup{max}}$ steps while satisfying the following inequality (cf. \eqref{ctrb_argument_t})
	\begin{equation}
		\begin{aligned}
			\norm*{\mathbf{\bar{u}}'_{[L-d_{\textup{max}},L-1]}(t+d_{\textup{max}})}_2^2 &\leq K_\gamma \norm*{\begin{bmatrix}
					\mathbf{v}_{[t+L,t+L+d_{\textup{max}}-1]}\\
					\bar{\Xi}'_{[L-d_{\textup{max}},L-1]}(t+d_{\textup{max}})
			\end{bmatrix}}_2^2 \leq K_\gamma \Gamma_{v}\norm*{\breve{\Xi}_{t+L}}_2^2.
		\end{aligned}\label{ctrb_argument_tdmax}
	\end{equation}
	\noindent Notice that the effect of $\mathbf{\bar{u}}'_{[L-d_{\textup{max}},L-d_i-1]}(t+d_{\textup{max}})$ (and the corresponding $\mathbf{v}_{[t+L,t+L+d_{\textup{max}}-d_i-1]}$), appears at the $i-$th output $d_i$ instances afterwards. However, it only affects those outputs for which $d_i\neq d_{\textup{max}}$. In particular
	\begin{equation}
		\begin{bmatrix}
			\bar{y}_{i,L+d_i-d_{\textup{max}}}'(t+d_{\textup{max}})\\
			\vdots\\
			\bar{y}_{i,L-1}'(t+d_{\textup{max}})					
		\end{bmatrix}
		 = {v}_{i,[t+L,t+L+d_{\textup{max}}-d_i-1]}.\label{pre_y_di_not_dmax}
	\end{equation}
	Finally, the terminal constraints fix the last $d_i$ instances of each output according to \eqref{pc2_term}. To sum up, the candidate solutions for the input and output sequences at time $t+d_{\textup{max}}$ take the following form
	\begin{equation}
		\begin{aligned}
			\mathbf{\bar{u}}'_{[-d_{\textup{max}},-1]}(t+d_{\textup{max}}) &\stackrel{\eqref{pc2_ini}}{=} \mathbf{\bar{u}}^*_{[0,d_{\textup{max}}-1]}(t) = \mathbf{{u}}_{[t,t+d_{\textup{max}}-1]},\\
			\mathbf{\bar{u}}'_{[0,L-d_{\textup{max}}-1]}(t+d_{\textup{max}}) &\stackrel{\eqref{pre_u_candidate_tdmax}}{=} \mathbf{\bar{u}}^*_{[d_{\textup{max}},L-1]}(t),\\
			\mathbf{\bar{u}}'_{[L-d_{\textup{max}},L-1]}(t+d_{\textup{max}}) &: \textup{deadbeat controller that satisfies }\eqref{ctrb_argument_tdmax},\\
			\bar{y}'_{i,[-d_{\textup{max}},-1]}(t+d_{\textup{max}}) &\stackrel{\eqref{pc2_ini}}{=} \tilde{y}_{i,[t,t+d_{\textup{max}}-1]} = \underbrace{y_{i,[t,t+d_{\textup{max}}-1]}}_{=\breve{y}_{i,[t,t+d_{\textup{max}}-1]}} + w_{i,[t,t+d_{\textup{max}}-1]},\\
			\bar{y}'_{i,[0,L+d_i-d_{\textup{max}}-1]}(t+d_{\textup{max}}) &\stackrel{\eqref{pre_y_candidate_tdmax}}{=} \breve{y}_{i,[t+d_{\textup{max}},t+L+d_{i}-1]},\\
			\bar{y}_{i,[L+d_i-d_{\textup{max}},L-1]}'(t+d_{\textup{max}}) &\stackrel{\eqref{pre_y_di_not_dmax}}{=} {v}_{i,[t+L,t+L+d_{\textup{max}}-d_i-1]},\quad \textup{ for }d_i\neq d_{\textup{max}},\\
			\bar{y}_{i,[L,L+d_i-1]}'(t+d_{\textup{max}}) &\stackrel{\eqref{pc2_term}}{=} \mathbf{0}.
		\end{aligned}\label{cand_sols_tdmax}
	\end{equation}
	Next, we choose $\alpha'(t+d_{\textup{max}})$ as
	\begin{gather}
		\alpha'(t+d_{\textup{max}}) = \begin{bmatrix}
			H_{L+d_{\textup{max}}}(\hat{\Psi}(\mathbf{u}^{\textup{d}},\tilde{\Xi}^{\textup{d}}) + \hat{E}(\mathbf{u}^{\textup{d}},\tilde{\Xi}^{\textup{d}}) + \hat{D}(\omega^{\textup{d}}) ) \\ H_1(\Xi_{[0,N-L-d_{\textup{max}}]}^{\textup{d}})
		\end{bmatrix}^{\dagger}\begin{bmatrix}
		\hat{\Psi}(\mathbf{\bar{u}}'(t+d_{\textup{max}}),\breve{\Xi}) + \hat{E}(\mathbf{\bar{u}}'(t+d_{\textup{max}}),\breve{\Xi}) + \hat{D}(\breve{\omega})\\ {\Xi}_{t}
	\end{bmatrix},\label{candidate_alpha_tdmax}
	\end{gather}
	where $\breve{\Xi}\coloneqq\breve{\Xi}_{[t,t+L-1]}$ with each $\breve{\Xi}_{t+j} = \begin{bmatrix}\breve{y}_{1,[t+j,t+j+d_1-1]}^\top & \cdots& \breve{y}_{m,[t+j,t+j+d_m-1]}^\top\end{bmatrix}^\top$ and $\breve{\omega}\coloneqq\begin{bmatrix}{\omega}_{[t,t+d_{\textup{max}}-1]}^\top & \mathbf{0}\end{bmatrix}^\top$.\par
	The choice of $\alpha'(t+d_{\textup{max}})$ exists by persistency of excitation of $\{\hat{\Psi}_k(\mathbf{u}^{\textup{d}},\tilde{\Xi}^{\textup{d}}) + \hat{E}_k(\mathbf{u}^{\textup{d}},\tilde{\Xi}^{\textup{d}}) + \hat{D}_k(\omega^{\textup{d}})\}_{k=0}^{N-1}$ (cf. Assumption \ref{PE_asmp}). Furthermore, this choice of $\alpha'(t+d_{\textup{max}})$ satisfies
	\begin{equation*}
		\begin{bmatrix}
			H_{L+d_{\textup{max}}}(\hat{\Psi}(\mathbf{u}^{\textup{d}},\tilde{\Xi}^{\textup{d}}) + \hat{E}(\mathbf{u}^{\textup{d}},\tilde{\Xi}^{\textup{d}}) + \hat{D}(\omega^{\textup{d}}) ) \\ H_{L+d_{\textup{max}}+1}(\Xi^{\textup{d}})
		\end{bmatrix}\alpha'(t+d_{\textup{max}}) = \begin{bmatrix}
			\hat{\Psi}(\mathbf{\bar{u}}'(t),\breve{\Xi}) + \hat{E}(\mathbf{\bar{u}}'(t),\breve{\Xi}) + \hat{D}(\breve{\omega})\\ \begin{bmatrix}\breve{\Xi}_{[t-d_{\textup{max}},t+L-1]}^\top&\mathbf{0}\end{bmatrix}^\top
		\end{bmatrix}.
	\end{equation*}
	Finally, by plugging \eqref{candidate_alpha_tdmax} into \eqref{pc2_willems}, we get the following expression for $\sigma'(t+d_{\textup{max}})$
	\begin{align}
		&\sigma'(t+d_{\textup{max}}) = \begin{bmatrix}
			\sigma'_{\Psi}(t+d_{\textup{max}}) \\ \sigma'_{\Xi}(t+d_{\textup{max}})
		\end{bmatrix} =\label{candidate_sigma_tdmax}
	\\&\scalebox{0.95}{$\begin{bmatrix}
			\hat{\Psi}(\mathbf{\bar{u}}'(t+d_{\textup{max}}),\breve{\Xi}) - \hat{\Psi}(\mathbf{\bar{u}}'(t+d_{\textup{max}}),\bar{\Xi}'(t+d_{\textup{max}})) + \hat{E}(\mathbf{\bar{u}}'(t+d_{\textup{max}}),\breve{\Xi}) + \hat{D}(\breve{\omega}) - H_{L+d_{\textup{max}}}( \hat{E}(\mathbf{u}^{\textup{d}},\tilde{\Xi}^{\textup{d}}) + \hat{D}(\omega^{\textup{d}}) )\alpha'(t+d_{\textup{max}}) \\
			H_{L+d_{\textup{max}}+1}(\omega^{\textup{d}})\alpha'(t+d_{\textup{max}}) - \begin{bmatrix}
				\omega_{[t, t+d_{\textup{max}}-1]}\\ \mathbf{0}\\ \mathbf{0}
			\end{bmatrix}
		\end{bmatrix}$}\notag
	\end{align}
	which satisfies \eqref{slack_const} (this can be shown by following similar arguments as in \eqref{candidate_sigma_at_t_satisfies_slack_const}).
	In (\ref{cand_sols_tdmax}--\ref{candidate_sigma_tdmax}), candidate solutions were shown to exist and satisfy the constraints (\ref{pc2_willems}--\ref{slack_const}) at the subsequent iteration ($t+d_{\textup{max}}$) of the MPC scheme given that the previous iteration (at time $t$) was feasible, which completes the proof.
\end{proof}
Theorem \ref{recrusive_feasibility_theorem} showed that for any $V_{ROA}>0$, there exists $\bar{\varepsilon}>0$ and $\bar{w}>0$ such that for any $\varepsilon^*\leq\bar{\varepsilon},\,w^*\leq\bar{w}$ and any state $\Xi_t$ starting from the sub-level set $\mathbb{V}\coloneqq\{\Xi_t\in\mathbb{R}^n~|~V_t\leq V_{ROA}\}$, if the $d_{\textup{max}}-$step predictive control scheme \eqref{DDNMPC2} is feasible at time $t$, then it is feasible at time $t+d_{\textup{max}}$. Similar to \cite{Berberich203}, as the size of the sub-level set $\mathbb{V}$ increases, the term $\max\{\varepsilon^*, w^*\}$ decreases, and vice versa. This is explained by the fact that a larger sub-level set $\mathbb{V}$ includes states $\Xi_t$ that are farther away from the origin. Starting from such states makes the task of stabilizing the origin more difficult if $\max\{\varepsilon^*, w^*\}$ is large, and hence a small value of $\max\{\varepsilon^*, w^*\}$ is required for the results in Theorem \ref{recrusive_feasibility_theorem} to hold.\par
This does not yet imply recursive feasibility of the predictive control scheme for all time (cf. \cite{Berberich203}). To show that, we must show that the sub-level set $\mathbb{V}$ is invariant. This is presented in the next section along with the main result of practical stability.
	\section{Practical stability}\label{prac_stb_sec}
In the final section of this report, we combine all of the above arguments and present the main result of practical stability and exponential convergence of the Lyapunov function to a region whose size depends on $\varepsilon^*$ and $w^*$. The results shown below follow similar arguments as in \cite[Theorem 3]{Berberich203}, that are carefully adjusted for the nonlinear system \eqref{NLsys} under consideration.
\begin{theorem}\label{main_result_of_report}
	Let Assumptions \ref{bounded_err_assmp}--\ref{horizon_rob_asmp} hold. Then, for any $V_{ROA}>0$, there exist $\underline{\lambda}_\alpha,\bar{\lambda}_\alpha,\underline{\lambda}_\sigma,\bar{\lambda}_\sigma>0$ such that for all $\lambda_\alpha, \lambda_\sigma>0$ satisfying
	\begin{equation*}
		\underline{\lambda}_\alpha \leq \lambda_\alpha \leq \bar{\lambda}_\alpha, \qquad \underline{\lambda}_\sigma \leq \lambda_\sigma \leq \bar{\lambda}_\sigma,
	\end{equation*}
	there exists constants $\bar{\varepsilon},\bar{w},\bar{c}_{\textup{pe}}>0$ as well as a continuous, strictly increasing function $\beta:[0,\bar{\varepsilon}]\times[0,\bar{w}]\to[0,V_{ROA}]$ with $\beta(0,0)=0$, such that for all $\varepsilon^*,w^*, c_{\textup{pe}}$ satisfying
	\begin{equation*}
		\varepsilon^*\leq \bar{\varepsilon},\qquad w^*\leq\bar{w},\qquad c_{\textup{pe}}\max\{\varepsilon^*,w^*\}\leq \bar{c}_{\textup{pe}}
	\end{equation*}
	the sub-level set $\mathbb{V}$ is invariant and $V_t$ converges exponentially to $V_t\leq \beta(\varepsilon^*, w^*)$ in closed-loop with $d_{\textup{max}}-$step MPC scheme for all initial conditions for which $V_0\leq V_{ROA}$.
\end{theorem}
\begin{proof}
	Let $\varepsilon^*\leq\bar{\varepsilon},\,w^*\leq\bar{w}$ be small enough such that Theorem \ref{recrusive_feasibility_theorem} holds and let Problem \eqref{DDNMPC2} be feasible at time $t+d_{\textup{max}}$. Then by optimality,
	\begin{equation}
		\begin{aligned}
			J^*_{t+d_{\textup{max}}} &\leq J_{t+d_{\textup{max}}}\\
			&= \sum\limits_{k=0}^{L-1}\ell(\mathbf{\bar{u}}'_k(t+d_{\textup{max}}),\mathbf{\bar{y}}'_k(t+d_{\textup{max}})) + \lambda_\alpha\max\{\varepsilon^*, w^*\}\norm*{\alpha'(t+d_{\textup{max}})}_2^2+\lambda_\sigma\norm*{\sigma'(t+d_{\textup{max}})}_2^2\\
			&\quad+ J^*_t - \sum\limits_{k=0}^{L-1}\ell(\mathbf{\bar{u}}^*_k(t),\mathbf{\bar{y}}^*_k(t))-\lambda_\alpha\max\{\varepsilon^*, w^*\}\norm*{\alpha^*(t)}_2^2-\lambda_\sigma\norm*{\sigma^*(t)}_2^2,
		\end{aligned}\label{optimality_at_tdmax}
	\end{equation}
	where the last four terms were obtained by adding and subtracting the optimal value function at time $t$. We start by considering the summation terms in \eqref{optimality_at_tdmax}
	\begin{equation}
		\begin{aligned}
			\sum\limits_{k=0}^{L-1}\ell(\mathbf{\bar{u}}'_k(t+d_{\textup{max}}),\mathbf{\bar{y}}'_k(t+d_{\textup{max}})) - \sum\limits_{k=0}^{L-1}\ell(\mathbf{\bar{u}}^*_k(t),\mathbf{\bar{y}}^*_k(t)) &= \sum\limits_{k=0}^{L-d_{\textup{max}}-1}\ell\left(\mathbf{\bar{u}}'_k(t+d_{\textup{max}}),\mathbf{\bar{y}}'_k(t+d_{\textup{max}})\right) \\
			&\quad+ \sum\limits_{k=L-d_{\textup{max}}}^{L-1}\ell\left(\mathbf{\bar{u}}'_{k}(t+d_{\textup{max}}),\mathbf{\bar{y}}'_k(t+d_{\textup{max}})\right)\\
			&\quad-\sum\limits_{k=0}^{d_{\textup{max}}-1}\ell\left(\mathbf{\bar{u}}^*_k(t),\mathbf{\bar{y}}_k^*(t)\right) - \sum\limits_{k=d_{\textup{max}}}^{L-1}\ell\left(\mathbf{\bar{u}}^*_k(t),\mathbf{\bar{y}}_k^*(t)\right).
		\end{aligned}
	\end{equation}
Notice that by rearranging the summation limits, one obtains
	\begin{equation}
		\sum\limits_{k=d_{\textup{max}}}^{L-1}\ell\left(\mathbf{\bar{u}}^*_k(t),\mathbf{\bar{y}}_k^*(t)\right) = \sum\limits_{k=0}^{L-d_{\textup{max}}-1}\ell\left(\mathbf{\bar{u}}^*_{k+d_{\textup{max}}}(t),\mathbf{\bar{y}}^*_{k+d_{\textup{max}}}(t)\right).
	\end{equation}
Therefore
	\begin{equation}
		\begin{aligned}
			\sum\limits_{k=0}^{L-1}\ell(\mathbf{\bar{u}}'_k(t+d_{\textup{max}})&,\mathbf{\bar{y}}'_k(t+d_{\textup{max}})) - \sum\limits_{k=0}^{L-1}\ell(\mathbf{\bar{u}}^*_k(t),\mathbf{\bar{y}}^*_k(t))\\
			&=\sum\limits_{k=0}^{L-d_{\textup{max}}-1}\left(\ell\left(\mathbf{\bar{u}}'_k(t+d_{\textup{max}}),\mathbf{\bar{y}}'_k(t+d_{\textup{max}})\right) - \ell\left(\mathbf{\bar{u}}^*_{k+d_{\textup{max}}}(t),\mathbf{\bar{y}}^*_{k+d_{\textup{max}}}(t)\right)\right)\\
			&\quad+\sum\limits_{k=L-d_{\textup{max}}}^{L-1}\ell\left(\mathbf{\bar{u}}'_{k}(t+d_{\textup{max}}),\mathbf{\bar{y}}'_k(t+d_{\textup{max}})\right) - \sum\limits_{k=0}^{d_{\textup{max}}-1}\ell\left(\mathbf{\bar{u}}^*_k(t),\mathbf{\bar{y}}_k^*(t)\right).
		\end{aligned}
	\label{sum_terms_at_tdmax}
	\end{equation}
Notice that the inputs in the first summation term on the RHS of \eqref{sum_terms_at_tdmax} are the same, i.e., $\mathbf{\bar{u}}'_{[0,L-d_{\textup{max}}-1]}(t+d_{\textup{max}}) = \mathbf{\bar{u}}^*_{[d_{\textup{max}},L-1]}(t)$ by \eqref{cand_sols_tdmax}. Therefore, the stage costs in that term are simply
\begin{equation}
	\begin{aligned}
		\ell(\mathbf{\bar{u}}'_k(t+d_{\textup{max}})&,\mathbf{\bar{y}}'_k(t+d_{\textup{max}})) - \ell\left(\mathbf{\bar{u}}^*_{k+d_{\textup{max}}}(t),\mathbf{\bar{y}}^*_{k+d_{\textup{max}}}(t)\right)\\
		& = \norm*{\mathbf{\bar{y}}'_k(t+d_{\textup{max}})}_Q^2 - \norm*{\mathbf{\bar{y}}^*_{k+d_{\textup{max}}}(t)}_Q^2\\
		&=\norm*{\mathbf{\bar{y}}'_k(t+d_{\textup{max}}) - \mathbf{\bar{y}}^*_{k+d_{\textup{max}}}(t) + \mathbf{\bar{y}}^*_{k+d_{\textup{max}}}(t) }_Q^2 - \norm*{\mathbf{\bar{y}}^*_{k+d_{\textup{max}}}(t)}_Q^2\\
		&= \norm*{\mathbf{\bar{y}}'_k(t+d_{\textup{max}}) - \mathbf{\bar{y}}^*_{k+d_{\textup{max}}}(t)}_Q^2 + 2(\mathbf{\bar{y}}'_k(t+d_{\textup{max}}) - \mathbf{\bar{y}}^*_{k+d_{\textup{max}}}(t))^\top Q \mathbf{\bar{y}}^*_{k+d_{\textup{max}}}(t)\\
		&\leq \norm*{\mathbf{\bar{y}}'_k(t+d_{\textup{max}}) - \mathbf{\bar{y}}^*_{k+d_{\textup{max}}}(t)}_Q^2 + 2\norm*{\mathbf{\bar{y}}'_k(t+d_{\textup{max}}) - \mathbf{\bar{y}}^*_{k+d_{\textup{max}}}(t)}_Q\norm*{\mathbf{\bar{y}}^*_{k+d_{\textup{max}}}(t)}_Q.
	\end{aligned}
\end{equation}
Since $J^*_t\leq V(t)\leq V_{ROA}$, it holds that $\norm*{\mathbf{\bar{y}}^*_{k+d_{\textup{max}}}(t)}_Q^2\leq J^*_t \leq V_{ROA}$, and using the inequality $2a \leq 1+a^2$ for $a\in\mathbb{R}$ one can write $ 2\norm*{\mathbf{\bar{y}}^*_{k+d_{\textup{max}}}(t)}_Q \leq 1+\norm*{\mathbf{\bar{y}}^*_{k+d_{\textup{max}}}(t)}_Q^2 \leq 1+V_{ROA}$. Therefore, an upper bound on the first summation term on the RHS of \eqref{sum_terms_at_tdmax} is expressed as
\begin{equation}
	\begin{aligned}
		&\sum\limits_{k=0}^{L-d_{\textup{max}}-1}\ell\left(\mathbf{\bar{u}}'_k(t+d_{\textup{max}}),\mathbf{\bar{y}}'_k(t+d_{\textup{max}})\right) - \ell\left(\mathbf{\bar{u}}^*_{k+d_{\textup{max}}}(t),\mathbf{\bar{y}}^*_{k+d_{\textup{max}}}(t)\right)\\
		&\leq \sum\limits_{k=0}^{L-d_{\textup{max}}-1}\left(\norm*{\mathbf{\bar{y}}'_k(t+d_{\textup{max}}) - \mathbf{\bar{y}}^*_{k+d_{\textup{max}}}(t)}_Q^2 + \norm*{\mathbf{\bar{y}}'_k(t+d_{\textup{max}}) - \mathbf{\bar{y}}^*_{k+d_{\textup{max}}}(t)}_Q(1+V_{ROA})\right).
	\end{aligned}\label{pre_1st_sum_term_tdmax_ub}
\end{equation}
Note that, for $k\in\mathbb{Z}_{[0,L-d_{\textup{max}}-1]}$, the following holds
\begin{equation}
	\norm*{\mathbf{\bar{y}}'_k(t+d_{\textup{max}}) - \mathbf{\bar{y}}^*_{k+d_{\textup{max}}}(t)}_Q^2 = \norm*{\begin{bmatrix}
			{\bar{y}}'_{1,k}(t+d_{\textup{max}}) - {\bar{y}}^*_{1,k+d_{\textup{max}}}(t) \\ \vdots \\ {\bar{y}}'_{m,k}(t+d_{\textup{max}}) - {\bar{y}}^*_{m,k+d_{\textup{max}}}(t)
	\end{bmatrix}}_Q^2. \label{y_difference_in_1st_sum_term_tdmax}
\end{equation}
Furthermore, note from \eqref{cand_sols_tdmax} that ${\bar{y}}'_{i,[0,L-d_{\textup{max}}-1]}(t+d_{\textup{max}})=\breve{y}_{i,[t+d_{\textup{max}},t+L-1]}$. Therefore, one can re-write \eqref{y_difference_in_1st_sum_term_tdmax} as
\begin{equation}
	\begin{aligned}
		&\norm*{\mathbf{\bar{y}}'_k(t+d_{\textup{max}}) - \mathbf{\bar{y}}^*_{k+d_{\textup{max}}}(t)}_Q^2\\
		&= \norm*{\begin{bmatrix}
				\breve{y}_{1,t+k+d_{\textup{max}}} - {\bar{y}}^*_{1,k+d_{\textup{max}}}(t) \\ \vdots \\ \breve{y}_{m,t+k+d_{\textup{max}}} - {\bar{y}}^*_{m,k+d_{\textup{max}}}(t)
		\end{bmatrix}}_Q^2 \leq \lambda_{\max}(Q)\norm*{\begin{bmatrix}
				\breve{y}_{1,t+k+d_{\textup{max}}} - {\bar{y}}^*_{1,k+d_{\textup{max}}}(t) \\ \vdots \\ \breve{y}_{m,t+k+d_{\textup{max}}} - {\bar{y}}^*_{m,k+d_{\textup{max}}}(t)
		\end{bmatrix}}_2^2\\
		&\leq m\lambda_{\max}(Q)\norm*{\begin{bmatrix}
				\breve{y}_{1,t+k+d_{\textup{max}}} - {\bar{y}}^*_{1,k+d_{\textup{max}}}(t) \\ \vdots \\ \breve{y}_{m,t+k+d_{\textup{max}}} - {\bar{y}}^*_{m,k+d_{\textup{max}}}(t)
		\end{bmatrix}}_\infty^2 \hspace{-3mm} = m\lambda_{\max}(Q)\max\limits_i\left|\breve{y}_{i,t+k+d_{\textup{max}}} - {\bar{y}}^*_{i,k+d_{\textup{max}}}(t)\right|^2\\
		& \stackrel{\eqref{output_difference}}{\leq} m\lambda_{\max}(Q)\max\limits_i\left[\mathcal{P}^{k + 2d_{\textup{max}} - d_i}(K_{\Xi})\Big(\varepsilon^*(1+\norm*{\alpha^*(t)}_1) + (1 + K_w)w^*\norm*{\alpha^*(t)}_1 + (1+\norm*{\mathcal{G}}_\infty)\norm*{\sigma^*(t)}_\infty\Big)\right]^2.
	\end{aligned}\label{y_difference_ub_in_1st_sum_term_tdmax}
\end{equation}
From \eqref{y_difference_ub_in_1st_sum_term_tdmax}, it also holds that
\begin{align}
	&\norm*{\mathbf{\bar{y}}'_k(t+d_{\textup{max}}) - \mathbf{\bar{y}}^*_{k+d_{\textup{max}}}(t)}_Q\notag\\
	&\leq \sqrt{m\lambda_{\max}(Q)}\max\limits_i\left[\mathcal{P}^{k + 2d_{\textup{max}} - d_i}(K_{\Xi})\Big(\varepsilon^*(1+\norm*{\alpha^*(t)}_1) + (1 + K_w)w^*\norm*{\alpha^*(t)}_1 + (1+\norm*{\mathcal{G}}_\infty)\norm*{\sigma^*(t)}_\infty\Big)\right],\label{y_difference_sqrt_ub_in_1st_sum_term_tdmax}
\end{align}
for $k\in\mathbb{Z}_{[0,L-d_{\textup{max}}-1]}$. Substituting \eqref{y_difference_ub_in_1st_sum_term_tdmax} and \eqref{y_difference_sqrt_ub_in_1st_sum_term_tdmax} in \eqref{pre_1st_sum_term_tdmax_ub} we obtain
\begin{equation}
	\begin{aligned}
		&\sum\limits_{k=0}^{L-d_{\textup{max}}-1}\ell\left(\mathbf{\bar{u}}'_k(t+d_{\textup{max}}),\mathbf{\bar{y}}'_k(t+d_{\textup{max}})\right) - \ell\left(\mathbf{\bar{u}}^*_{k+d_{\textup{max}}}(t),\mathbf{\bar{y}}^*_{k+d_{\textup{max}}}(t)\right)\\
		&\leq \sum\limits_{k=0}^{L-d_{\textup{max}}-1}\left\lbrace m\lambda_{\max}(Q)\max\limits_i\left[\mathcal{P}^{k + 2d_{\textup{max}} - d_i}(K_{\Xi})\Big(\varepsilon^*(1+\norm*{\alpha^*(t)}_1) + (1 + K_w)w^*\norm*{\alpha^*(t)}_1 + (1+\norm*{\mathcal{G}}_\infty)\norm*{\sigma^*(t)}_\infty\Big)\right]^2 \right.\\
		& \quad\left. + \sqrt{m\lambda_{\max}(Q)}\max\limits_i\left[\mathcal{P}^{k + 2d_{\textup{max}} - d_i}(K_{\Xi})\Big(\varepsilon^*(1+\norm*{\alpha^*(t)}_1) + (1 + K_w)w^*\norm*{\alpha^*(t)}_1 + (1+\norm*{\mathcal{G}}_\infty)\norm*{\sigma^*(t)}_\infty\Big)\right](1+V_{ROA}) \right\rbrace.
	\end{aligned}\label{1st_sum_term_tdmax_ub}
\end{equation}
	The second summation term in \eqref{sum_terms_at_tdmax} can be bounded by
	\begin{equation}
		\begin{aligned}
			\sum\limits_{k=L-d_{\textup{max}}}^{L-1}\ell\left(\mathbf{\bar{u}}'_{k}(t+d_{\textup{max}}),\mathbf{\bar{y}}'_k(t+d_{\textup{max}})\right) &\leq \lambda_{\textup{max}}(Q,R)\norm*{\begin{bmatrix}
					\mathbf{\bar{u}}'_{[L-d_{\textup{max}},L-1]}(t+d_{\textup{max}})\\
					\mathbf{\bar{y}}'_{[L-d_{\textup{max}},L-1]}(t+d_{\textup{max}})
			\end{bmatrix}}_2^2\\ &= \lambda_{\textup{max}}(Q,R)\norm*{\begin{bmatrix}
			\mathbf{\bar{u}}'_{[L-d_{\textup{max}},L-1]}(t+d_{\textup{max}})\\
			{\bar{y}}'_{1,[L-d_{\textup{max}},L+d_1-2]}(t+d_{\textup{max}})\\
			\vdots\\
			{\bar{y}}'_{m,[L-d_{\textup{max}},L+d_m-2]}(t+d_{\textup{max}})
		\end{bmatrix}}_2^2
		\end{aligned}\label{2nd_sum_term_at_tdmax}
	\end{equation}
	where the last equality comes from appending $\bar{y}'_{i,[L,L+d_i-2]}(t+d_{\textup{max}})\stackrel{\eqref{pc2_term}}{=}\mathbf{0}$. By definition of $\Xi_k$ from \eqref{definition_of_plain_Xi}, one can write
	\begin{equation}
		\begin{aligned}
			\sum\limits_{k=L-d_{\textup{max}}}^{L-1}\ell\left(\mathbf{\bar{u}}'_{k}(t+d_{\textup{max}}),\mathbf{\bar{y}}'_k(t+d_{\textup{max}})\right) &\leq
			\lambda_{\textup{max}}(Q,R)\norm*{\begin{bmatrix}
					\mathbf{\bar{u}}'_{[L-d_{\textup{max}},L-1]}(t+d_{\textup{max}})\\
					{\Xi}'_{[L-d_{\textup{max}},L-1]}(t+d_{\textup{max}})
			\end{bmatrix}}_2^2\\
			&= \lambda_{\textup{max}}(Q,R)\underbrace{\norm*{
					\mathbf{\bar{u}}'_{[L-d_{\textup{max}},L-1]}(t+d_{\textup{max}})}_2^2}_{\stackrel{\eqref{ctrb_argument_tdmax}}{\leq} K_\gamma \Gamma_v \norm*{\breve{\Xi}_{t+L}}_2^2 } + \underbrace{ \norm*{
					{\Xi}'_{[L-d_{\textup{max}},L-1]}(t+d_{\textup{max}})}_2^2}_{\stackrel{\eqref{ctrb_argument_tdmax}}{\leq} \Gamma_v \norm*{\breve{\Xi}_{t+L}}_2^2 }\\
			&\leq \lambda_{\max}(Q,R)(K_\gamma+1)\Gamma_{v}\norm*{\breve{\Xi}_{t+L}}_2^2\\
		\end{aligned}\label{pre_2nd_sum_term_at_tdmax_ub}
	\end{equation}
	Recall that $\bar{y}_{i,[L,L+d_i-1]}^*(t)=0$ are fixed by the terminal constraints of the previous MPC iteration, and hence, by Lemma \ref{lemma_output_difference}, the instances $\breve{y}_{i,[t+L,t+L+d_i-1]}$ for all $i\in\mathbb{Z}_{[1,m]}$ satisfy the following inequality for all $k\in\mathbb{Z}_{[L,L+d_i-1]}$
	\begin{equation*}
		\left|\breve{y}_{i,t+k}\right| \leq \mathcal{P}^{k}(K_{\Xi})\Big(\varepsilon^*(1+\norm*{\mathcal{G}}_\infty) + (1 + K_w)w^*\norm*{\alpha^*(t)}_1 + (1+\norm*{g_i}_1)\norm*{\sigma^*(t)}_\infty\Big).
	\end{equation*}
	Therefore, $\norm*{\breve{\Xi}_{t+L}}_2^2$ in \eqref{pre_2nd_sum_term_at_tdmax_ub} can be expressed as
	\begin{equation}
		\begin{aligned}
			\norm*{\breve{\Xi}_{t+L}}_2^2&=\norm*{\begin{bmatrix}\breve{y}_{1,[t+L,t+L+d_1-1]}\\\vdots\\\breve{y}_{m,[t+L,t+L+d_m-1]}\end{bmatrix}}_2^2\\
			&\leq m \norm*{\begin{bmatrix}\breve{y}_{1,[t+L,t+L+d_1-1]}\\\vdots\\\breve{y}_{m,[t+L,t+L+d_m-1]}\end{bmatrix}}_\infty^2\\
			&= m\max\limits_i\left|\breve{y}_{i,[t+L,t+L+d_i-1]}\right|^2\\
			&\stackrel{\eqref{output_difference}}{\leq} m \max\limits_i \left[\mathcal{P}^{L+d_{\textup{max}}-1}(K_{\Xi})\Big(\varepsilon^*(1+\norm*{\alpha^*(t)}_1) + (1 + K_w)w^*\norm*{\alpha^*(t)}_1 + (1+\norm*{\mathcal{G}}_\infty)\norm*{\sigma^*(t)}_\infty\Big)\right]^2,
		\end{aligned}
	\end{equation}
	and the following bound for \eqref{2nd_sum_term_at_tdmax} is obtained
		\begin{gather}
			\hspace{-120mm}\sum\limits_{k=L-d_{\textup{max}}}^{L-1}\hspace{-2mm}\ell\left(\mathbf{\bar{u}}'_{k}(t+d_{\textup{max}}),\mathbf{\bar{y}}'_k(t+d_{\textup{max}})\right) \leq \notag\\
			\hspace{5mm}\lambda_{\max}(Q,R)(K_\gamma+1)\Gamma_{v} m \max\limits_i \left[\mathcal{P}^{L+d_{\textup{max}}-1}(K_{\Xi})\Big(\varepsilon^*(1+\norm*{\alpha^*(t)}_1) + (1 + K_w)w^*\norm*{\alpha^*(t)}_1 + (1+\norm*{\mathcal{G}}_\infty)\norm*{\sigma^*(t)}_\infty\Big)\right]^2.\label{2nd_sum_term_at_tdmax_ub}
		\end{gather}
	Finally, substituting \eqref{2nd_sum_term_at_tdmax_ub} and \eqref{1st_sum_term_tdmax_ub} in \eqref{sum_terms_at_tdmax} results in
	\begin{align}
			&\sum\limits_{k=0}^{L-1}\left(\ell(\mathbf{\bar{u}}'_k(t+d_{\textup{max}}),\mathbf{\bar{y}}'_k(t+d_{\textup{max}}))-\ell(\mathbf{\bar{u}}^*_k(t),\mathbf{\bar{y}}^*_k(t))\right)\label{sum_terms_at_tdmax_ub}\\
			&\leq \sum\limits_{k=0}^{L-d_{\textup{max}}-1}\left\lbrace m\lambda_{\max}(Q)\max\limits_i\left[\mathcal{P}^{k+2d_{\textup{max}}-d_i}(K_{\Xi})\Big(\varepsilon^*(1+\norm*{\alpha^*(t)}_1) + (1 + K_w)w^*\norm*{\alpha^*(t)}_1 + (1+\norm*{\mathcal{G}}_\infty)\norm*{\sigma^*(t)}_\infty\Big)\right]^2 \right.\notag\\
			& \quad\left. + \sqrt{m\lambda_{\max}(Q)}\max\limits_i\left[\mathcal{P}^{k+2d_{\textup{max}}-d_i}(K_{\Xi})\Big(\varepsilon^*(1+\norm*{\alpha^*(t)}_1) + (1 + K_w)w^*\norm*{\alpha^*(t)}_1 + (1+\norm*{\mathcal{G}}_\infty)\norm*{\sigma^*(t)}_\infty\Big)\right](1+V_{ROA}) \right\rbrace\notag\\
			&\quad + \lambda_{\max}(Q,R)(K_\gamma+1)\Gamma_{v} m \max\limits_i \left[\mathcal{P}^{L+d_{\textup{max}}-1}(K_{\Xi})\Big(\varepsilon^*(1+\norm*{\alpha^*(t)}_1) + (1 + K_w)w^*\norm*{\alpha^*(t)}_1 + (1+\norm*{\mathcal{G}}_\infty)\norm*{\sigma^*(t)}_\infty\Big)\right]^2\notag\\
			&\quad - \sum\limits_{k=0}^{d_{\textup{max}}-1}\ell\left(\mathbf{\bar{u}}^*_k(t),\mathbf{\bar{y}}_k^*(t)\right).\notag
	\end{align}
	At this point, we have obtained upper bounds for the summation term of \eqref{optimality_at_tdmax}. Next, we wish to find suitable upper bounds for $\sigma'(t+d_{\textup{max}})$ and $\alpha'(t+d_{\textup{max}})$. For $\sigma'(t+d_{\textup{max}})$, one can use similar arguments as in \eqref{ub_sigma_t} to obtain
	\begin{equation}
		\begin{aligned}
			\norm*{\sigma'(t+d_{\textup{max}})}_2^2 &\leq 2\left((r+n)(L+d_{\textup{max}}) + n\right)\left((K_{\Psi}w^*)^2 + 2(\varepsilon^* + K_w w^*)^2\norm*{\mathcal{G}^\dagger}_\infty^2(1+\norm*{\alpha'(t+d_{\textup{max}})}_1^2)\right)
			%\left((r+n)(L+d_{\textup{max}}) + n\right)\left(2 (K_{\Psi}w^*)^2 + 2(\norm*{\mathcal{G}^\dagger}_\infty\varepsilon^* + K_w w^*)^2(2+2\norm*{\alpha'(t+d_{\textup{max}})}_1^2)\right).
		\end{aligned}\label{sigma_tdmax_ub}
	\end{equation}
	As for $\norm*{\alpha'(t+d_{\textup{max}})}_2^2$, we use the definition from \eqref{candidate_alpha_tdmax} to obtain
	\begin{equation}
		\begin{aligned}
			&\norm*{\alpha'(t+d_{\textup{max}})}_2^2 \leq \\
			&\quad\norm*{\begin{bmatrix}
					H_{L+d_{\textup{max}}}(\hat{\Psi}(\mathbf{u}^{\textup{d}},\tilde{\Xi}^{\textup{d}}) + \hat{E}(\mathbf{u}^{\textup{d}},\tilde{\Xi}^{\textup{d}}) + \hat{D}(\omega^{\textup{d}}) ) \\ H_1(\Xi_{[0,N-L-d_{\textup{max}}]}^{\textup{d}})
				\end{bmatrix}^{\dagger}}_2^2\norm*{\begin{bmatrix}
				\hat{\Psi}(\mathbf{\bar{u}}'(t+d_{\textup{max}}),\breve{\Xi}) + \hat{E}(\mathbf{\bar{u}}'(t+d_{\textup{max}}),\breve{\Xi}) + \hat{D}(\breve{\omega})\\ {\Xi}_{t}
			\end{bmatrix}}_2^2.
		\end{aligned}\label{pre_alpha_tdmax_ub}
	\end{equation}
	Using similar arguments\footnote{The difference from \eqref{how_we_now_bound_RHS_vector} is that here, we have $\Xi_t = \breve{\Xi}_t$ since $y_{i,[t,t+d_{\textup{max}}-1]}=\breve{y}_{i,[t,t+d_{\textup{max}}-1]}$ as in \eqref{cand_sols_tdmax}.} as in \eqref{how_we_now_bound_RHS_vector}, we can bound the rightmost term in \eqref{pre_alpha_tdmax_ub} by
	\begin{equation}
		\begin{aligned}
			\norm*{\begin{bmatrix}
					\hat{\Psi}(\mathbf{\bar{u}}'(t+d_{\textup{max}}),\breve{\Xi}) + \hat{E}(\mathbf{\bar{u}}'(t+d_{\textup{max}}),\breve{\Xi}) + \hat{D}(\breve{\omega})\\ {\Xi}_{t}
			\end{bmatrix}}_2^2 &\leq  \underbrace{4\mu^2 + 4r(L + d_{\textup{max}})\norm*{\mathcal{G}^\dagger}_\infty^2((\varepsilon^*)^2+(K_w w^*)^2)}_{=\mu_1} + \underbrace{8K_\Psi\Gamma_v}_{\coloneqq\mu_3}\norm*{\Xi_t}_2^2
		\end{aligned}\label{rewriting_RHS_alpha_tdmax}
	\end{equation}
	By plugging \eqref{rewriting_RHS_alpha_tdmax} back into \eqref{pre_alpha_tdmax_ub}, we get
	\begin{equation}
		\begin{aligned}
			\norm*{\alpha'(t+d_{\textup{max}})}_2^2 &\leq c_{\textup{pe}}\mu_1 + c_{\textup{pe}}\mu_3\norm*{\Xi_t}_2^2.
		\end{aligned}\label{alpha_tdmax_ub}
	\end{equation}
	At this point, we have provided suitable upper bounds for the candidate solutions at $t+d_{\textup{max}}$. Before plugging them back into \eqref{optimality_at_tdmax}, we use the following bound on $\sigma^*(t)$ from the MPC previous iteration (cf. \eqref{ub_sigma_t})
	\begin{equation}
		\norm*{\sigma^*(t)}_2^2\leq 2\left((r+n)(L+d_{\textup{max}}) + n\right)\left((K_{\Psi}w^*)^2 + 2(\varepsilon^* + K_w w^*)^2\norm*{\mathcal{G}^\dagger}_\infty^2(1+\norm*{\alpha^*(t)}_1^2)\right)\label{sigma_star_bound_used_in_J_difference}
	\end{equation}
	We can now plug (\ref{sum_terms_at_tdmax_ub}) into \eqref{optimality_at_tdmax} to obtain
	\begin{equation}
		\begin{aligned}
			&J^*_{t+d_{\textup{max}}}\\
			&\stackrel{\eqref{sum_terms_at_tdmax_ub}}{\leq} \sum\limits_{k=0}^{L-d_{\textup{max}}-1}\left\lbrace m\lambda_{\max}(Q)\max\limits_i\left[\mathcal{P}^{k+2d_{\textup{max}}-d_i}(K_{\Xi})\Big(\varepsilon^*(1+\norm*{\alpha^*(t)}_1) + (1 + K_w)w^*\norm*{\alpha^*(t)}_1 + (1+\norm*{\mathcal{G}}_\infty)\norm*{\sigma^*(t)}_\infty\Big)\right]^2 \right.\\
			& \quad\left. + \sqrt{m\lambda_{\max}(Q)}\max\limits_i\left[\mathcal{P}^{k+2d_{\textup{max}}-d_i}(K_{\Xi})\Big(\varepsilon^*(1+\norm*{\alpha^*(t)}_1) + (1 + K_w)w^*\norm*{\alpha^*(t)}_1 + (1+\norm*{\mathcal{G}}_\infty)\norm*{\sigma^*(t)}_\infty\Big)\right](1+V_{ROA}) \right\rbrace\\
			&\quad + \lambda_{\max}(Q,R)(K_\gamma+1)\Gamma_{v} m \max\limits_i \left[\mathcal{P}^{L+d_{\textup{max}}-1}(K_{\Xi})\Big(\varepsilon^*(1+\norm*{\alpha^*(t)}_1) + (1 + K_w)w^*\norm*{\alpha^*(t)}_1 + (1+\norm*{\mathcal{G}}_\infty)\norm*{\sigma^*(t)}_\infty\Big)\right]^2\\
			&\quad + J_t^* - \sum\limits_{k=0}^{d_{\textup{max}}-1}\ell\left(\mathbf{\bar{u}}^*_k(t),\mathbf{\bar{y}}_k^*(t)\right)+ \lambda_\alpha\max\{\varepsilon^*,w^*\} \left( \norm*{\alpha'(t+d_{\textup{max}})}_2^2 - \norm*{\alpha^*(t)}_2^2 \right) + \lambda_\sigma \left(\norm*{\sigma'(t+d_{\textup{max}})}_2^2 - \norm*{\sigma^*(t)}_2^2\right)
		\end{aligned}
	\end{equation}
	Further, we make use of the inequality $(a+b+c)^2 \leq 2a^2 + 4b^2 + 4c^2$ for $a,b,c\in\mathbb{R}$ and expand the above terms as follows
	\begin{equation}
		\begin{aligned}
			&J^*_{t+d_{\textup{max}}} \leq J_t^* + \hspace{-3mm}\overbrace{\sum\limits_{k=0}^{L-d_{\textup{max}}-1}\hspace{-3mm} 2m\lambda_{\max}(Q) \left(\mathcal{P}^{k+d_{\textup{max}}}(K_{\Xi})\right)^2}^{\coloneqq c_6}\Big(2(\varepsilon^*)^2(1+\norm*{\alpha^*(t)}_1)^2 + 2(1 + K_w)^2(w^*)^2\norm*{\alpha^*(t)}_1^2 + 2(1+\norm*{\mathcal{G}}_\infty)^2\norm*{\sigma^*(t)}_\infty^2\Big)\\
			&+ \hspace{-3mm}\overbrace{\sum\limits_{k=0}^{L-d_{\textup{max}}-1}\hspace{-3mm} \sqrt{m\lambda_{\max}(Q)} \mathcal{P}^{k+d_{\textup{max}}}(K_{\Xi})(1+V_{ROA})}^{\coloneqq c_7}\Big(\varepsilon^*(1+\norm*{\alpha^*(t)}_1) + (1 + K_w)w^*\norm*{\alpha^*(t)}_1 + (1+\norm*{\mathcal{G}}_\infty)\norm*{\sigma^*(t)}_\infty\Big)\\
			&+ \overbrace{2m\lambda_{\max}(Q,R)(K_\gamma+1)\Gamma_{v}(\mathcal{P}^{L+d_{\textup{max}}-1}(K_{\Xi}))^2}^{\coloneqq c_8}\Big(2(\varepsilon^*)^2(1+\norm*{\alpha^*(t)}_1)^2 + 2(1 + K_w)^2(w^*)^2\norm*{\alpha^*(t)}_1^2 + 2(1+\norm*{\mathcal{G}}_\infty)^2\norm*{\sigma^*(t)}_\infty^2\Big)\\
			&- \sum\limits_{k=0}^{d_{\textup{max}}-1}\ell\left(\mathbf{\bar{u}}^*_k(t),\mathbf{\bar{y}}_k^*(t)\right)+ \lambda_\alpha\max\{\varepsilon^*,w^*\} \left( \norm*{\alpha'(t+d_{\textup{max}})}_2^2 - \norm*{\alpha^*(t)}_2^2 \right) + \lambda_\sigma \left(\norm*{\sigma'(t+d_{\textup{max}})}_2^2 - \norm*{\sigma^*(t)}_2^2\right)
		\end{aligned}
	\end{equation}
	Next, we use the upper bounds in \eqref{sigma_tdmax_ub} and \eqref{alpha_tdmax_ub} and write
	\begin{equation}
		\begin{aligned}
			J^*_{t+d_{\textup{max}}} &\leq J_t^* + (c_6 + c_8)\Big(2(\varepsilon^*)^2(1+\norm*{\alpha^*(t)}_1)^2 + 2(1 + K_w)^2(w^*)^2\norm*{\alpha^*(t)}_1^2 + 2(1+\norm*{\mathcal{G}}_\infty)^2\norm*{\sigma^*(t)}_\infty^2\Big)\\
			&+ c_7\Big(\varepsilon^*(1+\norm*{\alpha^*(t)}_1) + (1 + K_w)w^*\norm*{\alpha^*(t)}_1 + (1+\norm*{\mathcal{G}}_\infty)\norm*{\sigma^*(t)}_\infty\Big)- \sum\limits_{k=0}^{d_{\textup{max}}-1}\ell\left(\mathbf{\bar{u}}^*_k(t),\mathbf{\bar{y}}_k^*(t)\right)\\
			&+ (\lambda_\alpha\max\{\varepsilon^*,w^*\} + c_9)c_{\textup{pe}}(\mu_1 + \mu_3\norm*{\Xi_t}_2^2) - \lambda_\alpha\max\{\varepsilon^*,w^*\}\norm*{\alpha^*(t)}_2^2 -\lambda_\sigma\norm*{\sigma^*(t)}_2^2 + c_{10}
		\end{aligned}\label{prepare_norm_equivalence}
	\end{equation}
where 
\begin{equation*}
	\begin{aligned}
		c_9 &\coloneqq 4\lambda_\sigma\left((r+n)(L+d_{\textup{max}}) + n\right)\left(\varepsilon^* + K_w w^*\right)^2\norm*{\mathcal{G}^\dagger}_\infty^2\\
		c_{10} &\coloneqq 2\lambda_\sigma\left((r+n)(L+d_{\textup{max}}) + n\right)\left( (K_\Psi w^*)^2 + 2(\varepsilon^* + K_w w^*)^2\right)\norm*{\mathcal{G}^\dagger}_\infty^2
	\end{aligned}
\end{equation*}
We will make use of the following relations which can be obtained by basic arithmetic and norm equivalences. First, notice that
\begin{equation*}
	\begin{aligned}
		(1+\norm*{\alpha^*(t)}_1)^2 &= (1+2\norm*{\alpha^*(t)}_1+\norm*{\alpha^*(t)}_1^2)\\
		&\leq 1+(1+\norm*{\alpha^*(t)}_1^2)+\norm*{\alpha^*(t)}_1^2\\
		&= 2(1+\norm{\alpha^*(t)}_1^2)\\
		&\leq 2(1+(N-L-d_{\textup{max}}+1)\norm*{\alpha^*(t)}_2^2) \coloneqq 2 + 2c_{11}\norm*{\alpha^*(t)}_2^2,
	\end{aligned}
\end{equation*}
where the second step was obtained by using the inequality $2a\leq 1+a^2$ for $a\in\mathbb{R}$ and the last step was obtained by norm equivalence. In particular, $\norm*{\alpha^*(t)}_1^2\leq c_{11}\norm*{\alpha^*(t)}_2^2$. Similarly, one can write
\begin{equation*}
	\begin{aligned}
		2\norm*{\alpha^*(t)}_1 &\leq 1+\norm*{\alpha^*(t)}_1^2\\
		\norm*{\alpha^*(t)}_1 &\leq 0.5+0.5\norm*{\alpha^*(t)}_1^2\\
		1+\norm*{\alpha^*(t)}_1 &\leq 1.5+0.5\norm*{\alpha^*(t)}_1^2\\
		1+\norm*{\alpha^*(t)}_1 &\leq 1.5+0.5c_{11}\norm*{\alpha^*(t)}_2^2.
	\end{aligned}
\end{equation*}
Also, notice that
\begin{equation*}
	\begin{aligned}
		\norm*{\sigma^*(t)}_\infty \stackrel{\eqref{slack_const}}{\leq} &K_{\Psi}w^* + (\varepsilon^* + K_w w^*)\norm*{\mathcal{G}^\dagger}_\infty(1+\norm*{\alpha^*(t)}_1)\\
		\leq\; &K_{\Psi}w^* + (\varepsilon^* + K_w w^*)\norm*{\mathcal{G}^\dagger}_\infty(1.5+0.5c_{11}\norm*{\alpha^*(t)}_2^2),
	\end{aligned}
\end{equation*}
and finally $\norm*{\sigma^*(t)}_\infty^2 \leq \norm*{\sigma^*(t)}_2^2$. Now, one can rewrite \eqref{prepare_norm_equivalence} as
\begin{equation}
	\begin{aligned}
		J^*_{t+d_{\textup{max}}} &\leq J_t^* + (c_6 + c_8)\Big(4(\varepsilon^*)^2(1+\norm*{\alpha^*(t)}_2^2) + 2(1 + K_w)^2(w^*)^2\norm*{\alpha^*(t)}_1^2 + 2(1+\norm*{\mathcal{G}}_\infty)^2\norm*{\sigma^*(t)}_2^2\Big)\\
		&+ c_7\Big(\varepsilon^*(1.5+0.5c_{11}\norm*{\alpha^*(t)}_2^2) + 0.5(1 + K_w)w^*(1+c_{11}\norm*{\alpha^*(t)}_2^2) \\
		& + (1+\norm*{\mathcal{G}}_\infty)\left(K_{\Psi}w^* + (\varepsilon^* + K_w w^*)\norm*{\mathcal{G}^\dagger}_\infty(1.5+0.5c_{11}\norm*{\alpha^*(t)}_2^2)\right)\Big) - \sum\limits_{k=0}^{d_{\textup{max}}-1}\ell\left(\mathbf{\bar{u}}^*_k(t),\mathbf{\bar{y}}_k^*(t)\right)\\
		&+ (\lambda_\alpha\max\{\varepsilon^*,w^*\} + c_9)c_{\textup{pe}}(\mu_1 + \mu_3\norm*{\Xi_t}_2^2) - \lambda_\alpha\max\{\varepsilon^*,w^*\}\norm*{\alpha^*(t)}_2^2 -\lambda_\sigma\norm*{\sigma^*(t)}_2^2 + c_{10}.
	\end{aligned}\label{cancellingout_terms}
\end{equation}
This can be compactly written as
\begin{equation}
	\begin{aligned}
		J^*_{t+d_{\textup{max}}} - J_t^* &\leq - \sum\limits_{k=0}^{d_{\textup{max}}-1}\ell\left(\mathbf{\bar{u}}^*_k(t),\mathbf{\bar{y}}_k^*(t)\right)+ (\lambda_\alpha\max\{\varepsilon^*,w^*\} + c_9)c_{\textup{pe}}(\mu_1 + \mu_3\norm*{\Xi_t}_2^2) + c_{12}\\
		&\quad + (c_{13} - \lambda_\alpha\max\{\varepsilon^*,w^*\})\norm*{\alpha^*(t)}_2^2 + (c_{14} - \lambda_\sigma)\norm*{\sigma^*(t)}_2^2,
	\end{aligned}\label{J_difference_2}
\end{equation}
where 
\begin{equation*}
	\begin{aligned}
		c_{12} &\coloneqq (c_6+c_8)\left(4(\varepsilon^*)^2\right) + c_7\left(1.5\varepsilon^* + 0.5w^*(1+K_w) + (1+\norm*{\mathcal{G}}_\infty)\left(K_\Psi w^* + 1.5\left(\varepsilon^*+K_w w^*\right)\norm*{\mathcal{G}^\dagger}_\infty\right)\right) + c_{10}\\
		c_{13} &\coloneqq (c_6+c_8)\left(4(\varepsilon^*)^2 + 2(1+K_w)^2(w^*)^2\right) + 0.5c_7c_{11}\left(\varepsilon^* + (1+K_w)w^* + (1+\norm*{\mathcal{G}}_\infty)\left(\norm*{\mathcal{G}^\dagger}_\infty\varepsilon^*+K_w w^*\right)\right)\\
		c_{14} &\coloneqq 2(c_6+c_8)(1+\norm*{\mathcal{G}}_\infty)^2.
	\end{aligned}
\end{equation*}
Note that $c_{12}$ depends on both $\varepsilon^*$ and $w^*$ and vanishes as $\max\{\varepsilon^*,w^*\}\to0$.\par
	%%%%%%%%%%%%%%%%%%%%%%%%%%%%%%%%%%%%%%%%%%%%%%%%%%%%%%%%%%%%%%%%%%%%%%%%%%%%%%%%
	Now, the goal is to show that the Lyapunov function candidate satisfies a useful decay bound. The successive difference in the Lyapunov function is
	\begin{equation*}
		\begin{aligned}
			V_{t+d_{\textup{max}}} - V_t = \underbrace{J^*_{t+d_{\textup{max}}} - J^*_t}_{\textup{see }\eqref{J_difference_2}} + c_3(\underbrace{W_{t+d_{\textup{max}}}-W_t}_{\textup{see }\eqref{IOSS}}),
		\end{aligned}
	\end{equation*}
	which is bounded by
	\begin{equation}
		\begin{aligned}
			V_{t+d_{\textup{max}}} - V_t &\leq - \sum\limits_{k=0}^{d_{\textup{max}}-1}\ell\left(\mathbf{\bar{u}}^*_k(t),\mathbf{\bar{y}}_k^*(t)\right)+ (\lambda_\alpha\max\{\varepsilon^*,w^*\} + c_9)c_{\textup{pe}}(\mu_1 + \mu_3\norm*{\Xi_t}_2^2) + c_{12} \\
			&\quad + (c_{13} - \lambda_\alpha\max\{\varepsilon^*,w^*\})\norm*{\alpha^*(t)}_2^2 + (c_{14} - \lambda_\sigma)\norm*{\sigma^*(t)}_2^2\\
			&\quad + c_3\left(-\frac{1}{2}\norm*{\Xi_{[t,t+d_{\textup{max}}-1]}}_2^2 + {c}_1\norm*{\mathbf{v}_{[t,t+d_{\textup{max}}-1]}}_2^2 + {c}_2\norm*{\mathbf{y}_{[t,t+d_{\textup{max}}-1]}}_2^2\right).
		\end{aligned}\label{V_difference_1}
	\end{equation}
	Rewrite the term
	\begin{equation*}
		\begin{aligned}
			\norm*{\mathbf{y}_{[t,t+d_{\textup{max}}-1]} - \mathbf{\bar{y}}^*_{[0,d_{\textup{max}}-1]}(t) + \mathbf{\bar{y}}^*_{[0,d_{\textup{max}}-1]}(t)}_2^2 \leq 2\norm*{\mathbf{y}_{[t,t+d_{\textup{max}}-1]} - \mathbf{\bar{y}}^*_{[0,d_{\textup{max}}-1]}(t)}_2^2 + 2\norm*{\mathbf{\bar{y}}^*_{[0,d_{\textup{max}}-1]}(t)}_2^2,
		\end{aligned}
	\end{equation*}
	using $(a+b)^2\leq 2a^2+2b^2$ for $a,b\geq0$ and plug back into \eqref{V_difference_1} to get
	\begin{equation}
		\begin{aligned}
			V_{t+d_{\textup{max}}} - V_t &\leq - \sum\limits_{k=0}^{d_{\textup{max}}-1}\ell\left(\mathbf{\bar{u}}^*_k(t),\mathbf{\bar{y}}_k^*(t)\right)+ (\lambda_\alpha\max\{\varepsilon^*,w^*\} + c_9)c_{\textup{pe}}(\mu_1 + \mu_3\norm*{\Xi_t}_2^2) + c_{12}  \\
			&\quad + (c_{13} - \lambda_\alpha\max\{\varepsilon^*,w^*\})\norm*{\alpha^*(t)}_2^2 + (c_{14} - \lambda_\sigma)\norm*{\sigma^*(t)}_2^2\\
			&\quad + c_3\left(-\frac{1}{2}\norm*{\Xi_{[t,t+d_{\textup{max}}-1]}}_2^2 + {c}_1\norm*{\mathbf{v}_{[t,t+d_{\textup{max}}-1]}}_2^2\right)\\
			&\quad + 2c_2c_3\left(\norm*{\mathbf{y}_{[t,t+d_{\textup{max}}-1]} - \mathbf{\bar{y}}^*_{[0,d_{\textup{max}}-1]}(t)}_2^2 + \norm*{\mathbf{\bar{y}}^*_{[0,d_{\textup{max}}-1]}(t)}_2^2\right).
		\end{aligned}\label{V_difference_2}
	\end{equation}
	Recall from \eqref{cand_sols_tdmax} that $\mathbf{y}_{[t,t+d_{\textup{max}}-1]} = \mathbf{\breve{y}}_{[t,t+d_{\textup{max}}-1]}$.	This means that
	\begin{equation*}
		\begin{aligned}
			&\norm*{\mathbf{y}_{[t,t+d_{\textup{max}}-1]} - \mathbf{\bar{y}}^*_{[0,d_{\textup{max}}-1]}(t)}_2^2 = \norm*{\mathbf{\breve{y}}_{[t,t+d_{\textup{max}}-1]} - \mathbf{\bar{y}}^*_{[0,d_{\textup{max}}-1]}(t)}_2^2\\
			&= \norm*{\begin{bmatrix}
					\breve{y}_{1,[t,t+d_{\textup{max}}-1]} - {\bar{y}}^*_{1,[0,d_{\textup{max}}-1]}(t) \\ \vdots \\ \breve{y}_{m,[t,t+d_{\textup{max}}-1]} - {\bar{y}}^*_{m,[0,d_{\textup{max}}-1]}(t)
			\end{bmatrix}}_2^2\\
			&\leq m\max\limits_i\norm*{\tilde{y}_{i,[t,t+d_{\textup{max}}-1]} - {\bar{y}}^*_{i,[0,d_{\textup{max}}-1]}(t)}_\infty^2\\
			&\leq m \max\limits_i \left[\mathcal{P}^{2d_{\textup{max}}-d_i-1}(K_{\Xi})\Big(\varepsilon^*(1+\norm*{\alpha^*(t)}_1) + (1 + K_w)w^*\norm*{\alpha^*(t)}_1 + (1+\norm*{g_i}_1)\norm*{\sigma^*(t)}_\infty\Big)\right]^2\\
			&\leq m(\mathcal{P}^{d_{\textup{max}}-1}(K_{\Xi}))^2\Big(2(\varepsilon^*)^2(1+\norm*{\alpha^*(t)}_1)^2 + 2(1 + K_w)^2(w^*)^2\norm*{\alpha^*(t)}_1^2 + 2(1+\norm*{\mathcal{G}}_\infty)^2\norm*{\sigma^*(t)}_\infty^2\Big)\\
			&\leq m(\mathcal{P}^{d_{\textup{max}}-1}(K_{\Xi}))^2\Big(2(\varepsilon^*)^2(2+2c_{10}\norm*{\alpha^*(t)}_2^2) + 2(1 + K_w)^2(w^*)^2c_{10}\norm*{\alpha^*(t)}_2^2 \\
			& \hspace{35mm} + 2(1+\norm*{\mathcal{G}}_\infty)^2\left(K_{\Psi}w^* + (\varepsilon^* + K_w w^*)\norm*{\mathcal{G}^\dagger}_\infty(1.5+0.5c_{10}\norm*{\alpha^*(t)}_2^2)\right)^2\Big)\\
			&\coloneqq c_{15}\norm*{\alpha^*(t)}_2^2 + c_{16}
		\end{aligned}
	\end{equation*}
	where $c_{16}$ depends on $\varepsilon^*$ and $w^*$ and vanishes as $\max\{\varepsilon^*,w^*\}\to0$. Plugging back into \eqref{V_difference_2} results in
	\begin{equation}
		\begin{aligned}
			V_{t+d_{\textup{max}}} - V_t &\leq - \sum\limits_{k=0}^{d_{\textup{max}}-1}\ell\left(\mathbf{\bar{u}}^*_k(t),\mathbf{\bar{y}}_k^*(t)\right)+ (\lambda_\alpha\max\{\varepsilon^*,w^*\} + c_9)c_{\textup{pe}}(\mu_1 + \mu_3\norm*{\Xi_t}_2^2) + 2c_2c_3c_{16} + c_{12}  \\
			&\quad + (2c_2c_3c_{15} + c_{13} - \lambda_\alpha\max\{\varepsilon^*,w^*\})\norm*{\alpha^*(t)}_2^2 + (c_{14} - \lambda_\sigma)\norm*{\sigma^*(t)}_2^2\\
			&\quad + c_3\left(-\frac{1}{2}\norm*{\Xi_{[t,t+d_{\textup{max}}-1]}}_2^2 + {c}_1\norm*{\mathbf{v}_{[t,t+d_{\textup{max}}-1]}}_2^2\right) + 2c_2c_3\norm*{\mathbf{\bar{y}}^*_{[0,d_{\textup{max}}-1]}(t)}_2^2.
		\end{aligned}\label{V_difference_3}
	\end{equation}
By choosing $\lambda_\alpha,\,\lambda_\sigma$ sufficiently large and with $c_{\textup{pe}},\,\varepsilon^*,\,w^*$ being sufficiently small, one can cancel out the terms containing $\norm*{\alpha^*(t)}_2^2$ and $\norm*{\sigma^*(t)}_2^2$ on the RHS of \eqref{V_difference_3}. Therefore, we obtain
\begin{equation}
	\begin{aligned}
		V_{t+d_{\textup{max}}} - V_t &\leq - \sum\limits_{k=0}^{d_{\textup{max}}-1}\ell\left(\mathbf{\bar{u}}^*_k(t),\mathbf{\bar{y}}_k^*(t)\right)+ (\lambda_\alpha\max\{\varepsilon^*,w^*\} + c_9)c_{\textup{pe}}(\mu_1 + \mu_3\norm*{\Xi_t}_2^2) + 2c_2c_3c_{16} + c_{12}  \\
		&\quad + c_3\left(-\frac{1}{2}\norm*{\Xi_{[t,t+d_{\textup{max}}-1]}}_2^2 + {c}_1\norm*{\mathbf{v}_{[t,t+d_{\textup{max}}-1]}}_2^2\right) + 2c_2c_3\norm*{\mathbf{\bar{y}}^*_{[0,d_{\textup{max}}-1]}(t)}_2^2.
	\end{aligned}\label{V_difference_4}
\end{equation}
By local Lipschitz continuity of the function $\Phi(\cdot,\cdot)$, we can bound
\begin{equation*}
	\norm*{\mathbf{v}_{[t,t+d_{\textup{max}}-1]}}_2^2 \leq K_\Phi\left({\norm*{\mathbf{u}_{[t,t+d_{\textup{max}}-1]}}_2^2} + {\norm*{\Xi_{[t,t+d_{\textup{max}}-1]}}_2^2}\right) \stackrel{\eqref{pc2_ini}}{=} K_\Phi\left(\norm*{\mathbf{\mathbf{u}}^*_{[0,d_{\textup{max}}-1]}(t)}_2^2 + {\norm*{\Xi_{[t,t+d_{\textup{max}}-1]}}_2^2}\right).
\end{equation*}
Plugging back into \eqref{V_difference_4}, we obtain
\begin{equation}
	\begin{aligned}
		V_{t+d_{\textup{max}}} - V_t &\leq - \sum\limits_{k=0}^{d_{\textup{max}}-1}\ell\left(\mathbf{\bar{u}}^*_k(t),\mathbf{\bar{y}}_k^*(t)\right)+ (\lambda_\alpha\max\{\varepsilon^*,w^*\} + c_9)c_{\textup{pe}}(\mu_1 + \mu_3\norm*{\Xi_t}_2^2) + 2c_2c_3c_{16} + c_{12}  \\
		&\quad + c_3\left((c_1K_\Phi -0.5)\norm*{\Xi_{[t,t+d_{\textup{max}}-1]}}_2^2 + {c}_1K_\Phi\norm*{\mathbf{\mathbf{u}}^*_{[0,d_{\textup{max}}-1]}(t)}_2^2\right) + 2c_2c_3\norm*{\mathbf{\bar{y}}^*_{[0,d_{\textup{max}}-1]}(t)}_2^2\\
		&\stackrel{\eqref{ctrb_argument_t}}{\leq} - \sum\limits_{k=0}^{d_{\textup{max}}-1}\ell\left(\mathbf{\bar{u}}^*_k(t),\mathbf{\bar{y}}_k^*(t)\right)
		+ \Big((\lambda_\alpha\max\{\varepsilon^*,w^*\} + c_9)c_{\textup{pe}}\mu_3+c_3\Gamma_v(c_1K_\Phi-0.5)\Big)\norm*{\Xi_t}_2^2\\ &\quad+(\lambda_\alpha\max\{\varepsilon^*,w^*\} + c_9)c_{\textup{pe}}\mu_1 + 2c_2c_3c_{16} + c_{12}  + c_3\left( {c}_1K_\Phi\norm*{\mathbf{\mathbf{u}}^*_{[0,d_{\textup{max}}-1]}(t)}_2^2 + 2c_2\norm*{\mathbf{\bar{y}}^*_{[0,d_{\textup{max}}-1]}(t)}_2^2\right)\\
	\end{aligned}\label{V_difference_5}
\end{equation}
By choosing $c_3=\frac{\lambda_{\textup{min}}(Q,R)}{\max\{c_1 K_\Phi,\,2c_2\}}$ one can obtain the following
\begin{equation*}
	c_3\left({c}_1K_\Phi\norm*{\mathbf{\mathbf{u}}^*_{[0,d_{\textup{max}}-1]}(t)}_2^2 + 2c_2\norm*{\mathbf{\bar{y}}^*_{[0,d_{\textup{max}}-1]}(t)}_2^2\right) \leq \sum\limits_{k=0}^{d_{\textup{max}}-1}\left(\ell(\mathbf{\bar{u}}^*_k(t),\mathbf{\bar{y}}^*_k(t))\right),
\end{equation*}
which cancels out the first term on the RHS of \eqref{V_difference_5} and results in
\begin{equation}
	\begin{aligned}
		V_{t+d_{\textup{max}}} - V_t &\leq \Big((\lambda_\alpha\max\{\varepsilon^*,w^*\} + c_9)c_{\textup{pe}}\mu_3+c_3\Gamma_v(c_1K_\Phi-0.5)\Big)\norm*{\Xi_t}_2^2 + (\lambda_\alpha\max\{\varepsilon^*,w^*\} + c_9)c_{\textup{pe}}\mu_1 + 2c_2c_3c_{16} + c_{12}\\
		\Rightarrow V_{t+d_{\textup{max}}} - V_t &\leq (c_{17} - 0.5c_{18})\norm*{\Xi_t}_2^2 + c_{19},
	\end{aligned}\label{V_difference_6}
\end{equation}
where
\begin{equation}
	\begin{aligned}
		c_{17} &\coloneqq (\lambda_\alpha\max\{\varepsilon^*,w^*\} + c_9)c_{\textup{pe}}\mu_3+c_3\Gamma_vc_1K_\Phi,\\
		c_{18} &\coloneqq c_3\Gamma_v,\\
		c_{19} &\coloneqq (\lambda_\alpha\max\{\varepsilon^*,w^*\} + c_9)c_{\textup{pe}}\mu_1 + 2c_2c_3c_{16} + c_{12}.
	\end{aligned}\label{defs_c171819}
\end{equation}
We now construct the following function $\beta(\varepsilon^*, w^*)$. Notice that for any $\Xi_t\in\mathbb{B}_\delta$ with $V_t\leq V_{ROA}$, the following holds by \eqref{ub_V_t}
\begin{equation}
	V_t \leq \max\left\lbrace c_4,\,\frac{V_{ROA}-c_5}{\delta^2}\right\rbrace \norm*{\Xi_t}_2^2 + c_5.
\end{equation}
Similar to \cite{Berberich203}, we will start by building the function $\beta(\varepsilon^*, w^*)$ and show that $V_t$ converges to a neighborhood whose size depends on $\beta(\varepsilon^*, w^*)$ for a specific fixed choice of $V_{ROA}$. Then, we will generalize the result for any \textit{fixed} value of $V_{ROA}$. Let $V_{ROA}=\delta^2 c_4 + c_5$ and thus,
\begin{equation}
	V_t\leq c_4\norm*{\Xi_t}_2^2 + c_5
	\label{V_ub_fixed_VROA}
\end{equation}
Now we define the function $\beta(\varepsilon^*, w^*)$ as
\begin{equation}
	\beta(\varepsilon^*, w^*) \coloneqq \frac{c_4 c_{19} - c_5(c_{17} - 0.5 c_{18})}{(0.5 c_{18} - c_{17})},
	\label{beta_fcn}
\end{equation}
for any sufficiently small $\varepsilon^*$ and $w^*$ for which the denominator of \eqref{beta_fcn} is positive, i.e., $0.5 c_{18} - c_{17}>0$. Now we show that $\beta(\varepsilon^*, w^*)$ \textit{is} in fact positive if $\varepsilon^*$ is sufficiently small. Recall from the definitions of $c_{17},\,c_{18}$ in \eqref{defs_c171819} where
\begin{equation*}
	\begin{aligned}
		c_9 &= 4\lambda_\sigma\left((r+n)(L+d_{\textup{max}}) + n\right)\left(\varepsilon^* + K_w w^*\right)^2\norm*{\mathcal{G}^\dagger}_\infty^2,
	\end{aligned}
\end{equation*}
that if $\lambda_\alpha\leq\bar{\lambda}_\alpha,\,\lambda_\sigma\leq\bar{\lambda}_\sigma$ where $\bar{\lambda}_\alpha,\,\bar{\lambda}_\sigma$ are arbitrary but fixed upper bounds as in \cite{Berberich203}, and if $c_{\textup{pe}}\max\{\varepsilon^*,w^*\}\leq\bar{c}_{\textup{pe}},\,\varepsilon^*\leq\bar{\varepsilon},\,w^*\leq\bar{w}$ are sufficiently small, then $0.5 c_{18} - c_{17}>0$. Furthermore, the function $\beta(\varepsilon^*, w^*)$ has the following characteristics:
\begin{itemize}
	\item[] \textbf{(i)} It is continuous in $\varepsilon^*$ and $w^*$ by the definition of $c_i$ for $i\in\mathbb{Z}_{[1,19]}$.
	\item[] \textbf{(ii)} It satisfies $\beta(0,0)=0$. This can be seen from the definition of $c_i$ for $i\in\mathbb{Z}_{[1,19]}$.
	\item[] \textbf{(iii)} The function is strictly increasing since the numerator increases as $\max\{\varepsilon^*,w^*\}$ increases, while the denominator decreases (notice that the term $0.5c_3$ is independent of $\varepsilon^*,\,w^*$ and $c_{17}$ is negative), \textit{if} $0.5 c_{18} - c_{17}>0$ (which is the case for sufficiently small $\varepsilon^*,w^*$).
\end{itemize}
Next, we show invariance and exponential convergence of $V_t$ to \eqref{beta_fcn}. We start by considering \eqref{V_difference_6} and proceed with the following manipulations
\begin{equation}
	\begin{aligned}
		V_{t+d_{\textup{max}}} &\leq V_t + (c_{17} - 0.5c_{18})\norm*{\Xi_t}_2^2 + c_{19}\\
		V_{t+d_{\textup{max}}} &\leq V_t + (c_{17} - 0.5c_{18})\norm*{\Xi_t}_2^2 + c_{19} + \frac{1}{c_4}(c_{17}-0.5c_{18})V_t - \frac{1}{c_4}(c_{17}-0.5c_{18})V_t\\
		&\stackrel{\eqref{V_ub_fixed_VROA}}{\leq} \left( 1 + \frac{1}{c_4}(c_{17}-0.5c_{18}) \right)V_t + (c_{17} - 0.5c_{18})\norm*{\Xi_t}_2^2 + c_{19} - (c_{17} - 0.5c_{18})\norm*{\Xi_t}_2^2 - \frac{c_5}{c_4}(c_{17}-0.5c_{18})\\
		&= \left( 1 + \frac{1}{c_4}(c_{17}-0.5c_{18}) \right)V_t + c_{19} - \frac{c_5}{c_4}(c_{17}-0.5c_{18}) - \left( 1 + \frac{1}{c_4}(c_{17}-0.5c_{18}) \right)\beta(\varepsilon^*, w^*)\\
		&\quad + \left( 1 + \frac{1}{c_4}(c_{17}-0.5c_{18}) \right)\beta(\varepsilon^*, w^*)\\
		{V_{t+d_{\textup{max}}} - \beta(\varepsilon^*, w^*)}&\leq \left( 1 + \frac{1}{c_4}(c_{17}-0.5c_{18}) \right) (V_t-\beta(\varepsilon^*, w^*)) + c_{19} - \frac{c_5}{c_4}(c_{17}-0.5c_{18}) +\left( \frac{1}{c_4}(c_{17}-0.5c_{18}) \right)\beta(\varepsilon^*, w^*)\\
		V_{\beta, t+d_{\textup{max}}} &\stackrel{\eqref{beta_fcn}}{=} \left( 1 + \frac{1}{c_4}(c_{17}-0.5c_{18}) \right)(V_t-\beta(\varepsilon^*, w^*))+ c_{19} - \frac{c_5}{c_4}(c_{17}-0.5c_{18})\\
		&\quad +\left( \frac{1}{c_4}(c_{17}-0.5c_{18}) \right) \frac{c_4 c_{19} - c_5(c_{17} - 0.5 c_{18})}{(0.5 c_{18} - c_{17})}\\
	\end{aligned}
\end{equation}
The remaining terms cancel out as well and the result is the following contraction property
\begin{equation}
	V_{\beta, t+d_{\textup{max}}} \leq \underbrace{\left( 1 + \frac{1}{c_4}(c_{17}-0.5c_{18}) \right)}_{\stackrel{!}{<}1}(V_t-\beta(\varepsilon^*, w^*)).
	\label{final_contraction_result}
\end{equation}
To show that the factor above is in fact less than one and that the contraction holds, we recall that the denominator of $\beta(\varepsilon^*, w^*)$ in \eqref{beta_fcn} is strictly positive for $\lambda_\alpha\leq\bar{\lambda}_\alpha,\,\lambda_\sigma\leq\bar{\lambda}_\sigma$ and sufficiently small $\bar{c}_{\textup{pe}},\,\bar{\varepsilon},\,\bar{w}$. Therefore
\begin{equation*}
	\begin{aligned}
		0.5 c_{18} - c_{17} &> 0\\
		-0.5 c_{18} + c_{17} &< 0\\
		\frac{1}{c_4}\left(c_{17} - 0.5 c_{18} \right) &< 0,
	\end{aligned}
\end{equation*}
and hence the quantity under question in \eqref{final_contraction_result} is less than 1.\\\par
We have shown that for the specific choice of $V_{ROA}=\delta^2c_4+c_5$ the Lyapnuov function converges to $\beta(\varepsilon^*, w^*)$. However, the same can be shown for other \textit{fixed} choices of $V_{ROA}$ where $\beta(\varepsilon^*, w^*)$ takes the following form
\begin{equation*}
	\beta(\varepsilon^*,w^*) = \frac{\max\left\lbrace c_4,\frac{V_{ROA}-c_5}{\delta^2}\right\rbrace c_{19} - c_5(c_{17} - 0.5 c_{18})}{(0.5 c_{18} - c_{17})}.
\end{equation*}
By following similar arguments as above, it can be shown that the same contraction property in \eqref{final_contraction_result} holds for sufficiently small $\bar{c}_{\textup{pe}},\,\bar{\varepsilon},\,\bar{w}$ (cf. \cite{Berberich203}).
\end{proof}
	\ifCLASSOPTIONcaptionsoff
	\newpage
	\fi
	\bibliographystyle{ieeetr}
	\bibliography{refs}

\begin{thebibliography}{1}

\bibitem{Alsalti2022b}
M.~Alsalti, V.~G. Lopez, J.~Berberich, F.~Allgöwer, and M.~A. Müller,
  ``Data-driven nonlinear predictive control for feedback linearizable
  systems,'' {\em arXiv:2211.06339}, 2022.

\bibitem{Alsalti2022a}
M.~Alsalti, V.~G. Lopez, J.~Berberich, F.~Allgöwer, and M.~A. M\"uller,
  ``Data-based control of feedback linearizable systems,'' {\em IEEE
  Transactions on Automatic Control}, pp.~1--8, 2023.

\bibitem{Berberich203}
J.~Berberich, J.~Köhler, M.~A. Müller, and F.~Allgöwer, ``Data-driven model
  predictive control with stability and robustness guarantees,'' {\em IEEE
  Transactions on Automatic Control}, vol.~66, no.~4, pp.~1702--1717, 2021.

\bibitem{Monaco87}
S.~Monaco and D.~Normand-Cyrot, ``Minimum-phase nonlinear discrete-time systems
  and feedback stabilization,'' in {\em 26th {IEEE} Conference on Decision and
  Control}, {IEEE}, December 1987.

\bibitem{AlsaltiBerLopAll2021}
M.~Alsalti, J.~Berberich, V.~G. Lopez, F.~Allgöwer, and M.~A. Müller,
  ``Data-based system analysis and control of flat nonlinear systems,'' in {\em
  2021 60th IEEE Conference on Decision and Control (CDC)}, pp.~1484--1489,
  2021.

\bibitem{Coulson20}
J.~Coulson, J.~Lygeros, and F.~Dörfler, ``Data-enabled predictive control: In
  the shallows of the {D}ee{PC},'' in {\em 2019 18th European Control
  Conference (ECC)}, pp.~307--312, 2019.

\bibitem{Berberich205}
J.~Berberich, J.~Köhler, M.~A. Müller, and F.~Allgöwer, ``Robust constraint
  satisfaction in data-driven {MPC},'' in {\em 2020 59th IEEE Conference on
  Decision and Control (CDC)}, pp.~1260--1267, 2020.

\bibitem{Cai08}
C.~Cai and A.~R. Teel, ``Input–output-to-state stability for discrete-time
  systems,'' {\em Automatica}, vol.~44, no.~2, pp.~326--336, 2008.

\bibitem{Willems05}
J.~C. Willems, P.~Rapisarda, I.~Markovsky, and B.~L. {De Moor}, ``A note on
  persistency of excitation,'' {\em Systems \& Control Letters}, vol.~54,
  no.~4, pp.~325--329, 2005.

\end{thebibliography}
\end{document}